\documentclass[12pt,reqno]{amsart}

\usepackage{etex}
\reserveinserts{28}
\usepackage{nomencl}
\usepackage{appendix}
\usepackage{amsthm}
\usepackage{amscd}
\usepackage{amsfonts}
\usepackage{amssymb}
\usepackage{amsgen}
\usepackage{amsmath}
\usepackage{amsopn}
\usepackage{verbatim}
\usepackage{xypic}
\usepackage{xspace}
\usepackage{multicol}
\usepackage{eepic}
\usepackage{url}
\usepackage{upref}
\DeclareFontEncoding{OT2}{}{}
\DeclareFontSubstitution{OT2}{cmr}{m}{n}

\DeclareFontFamily{OT2}{cmr}{\hyphenchar\font45 }
\DeclareFontShape{OT2}{cmr}{m}{n}{%
   <5><6><7><8><9>gen*wncyr%
   <10><10.95><12><14.4><17.28><20.74><24.88>wncyr10}{}
\DeclareFontShape{OT2}{cmr}{b}{n}{%
   <5><6><7><8><9>gen*wncyb%
   <10><10.95><12><14.4><17.28><20.74><24.88>wncyb10}{}

\DeclareMathAlphabet{\mathcyr}{OT2}{cmr}{m}{n}
\DeclareMathAlphabet{\mathcyb}{OT2}{cmr}{b}{n}
\SetMathAlphabet{\mathcyr}{bold}{OT2}{cmr}{b}{n}

\theoremstyle{plain}
\newtheorem{thm}{Theorem}[section]
\newtheorem{lem}[thm]{Lemma}
\newtheorem{lem-defn}[thm]{Lemma-Definition}
\newtheorem{prop}[thm]{Proposition}
\newtheorem{cor}[thm]{Corollary}
\newtheorem{conj}[thm]{Conjecture}

\theoremstyle{definition}

\newtheorem{defn}[thm]{Definition}

\theoremstyle{remark}
\newtheorem{rem}[thm]{Remark}

 \newcommand{\nc}{\newcommand}

\renewcommand\Re{{\rm Re}}

\nc{\exref}[1]{\chapref\ref{#1}}

\nc{\per}[1]{\underset{#1}{\boldsymbol \pi}\,}
\nc{\QSym}{{\mathsf{QSym}}}
\nc{\tla}{\overset{\leftarrow}}
 \nc{\MT}{{\rm MT}}
 \nc{\XX}{{X}}
 \nc{\gF}{{\varPhi}}
 \nc{\gL}{{\Lambda}}
 \nc{\ot}{\otimes}

 \nc{\dual}{\ast}
 \nc{\wht}{\widehat}
 \nc{\bwg}{{\bigwedge}}
 \nc{\wg}{{\wedge}}
 \nc{\mmu}{{\boldsymbol{\mu}}}
 \nc{\mal}{{{\scriptstyle \maltese}}}
 \nc{\fA}{{\mathfrak A}}
 \nc{\hfA}{{\widehat{\mathfrak A}}}
 \nc{\HH}{{\mathfrak H}}
 \nc{\ra}{\rightarrow}
 \nc{\ors}{{\bfs}}
 \nc{\orr}{{\bfr}}
 \nc{\os}{{\overset}}
 \nc{\G}{{\mathbb G}}
 \nc{\F}{{\mathbb F}}
 \nc{\Z}{{\mathbb Z}}
 \nc{\R}{{\mathbb R}}
 \nc{\N}{{\mathbb N}}
 \nc{\ZN}{{{\mathbb N}_0}}
 \nc{\Q}{{\mathbb Q}}
 \nc{\CC}{{\mathbb C}}
 \nc{\V}{{\mathbb V}}
 \nc{\CP}{{\mathbb{CP}}}
 \nc{\Cnn}{{\mathbb C}_{\ge 0}}
 \nc{\Cp}{{\mathbb C}_{>0}}

 \nc{\MHSS}{MH${}^\star$S\xspace}
 \nc{\MHSSs}{MH${}^\star$Ss\xspace}
 \nc{\MZSV}{MZ${}^\star$V\xspace}
 \nc{\MZSVs}{MZ${}^\star$Vs\xspace}
 \nc{\FMZSV}{FMZ${}^\star$V\xspace}
 \nc{\FMZSVs}{FMZ${}^\star$Vs\xspace}
 \nc{\FESS}{FE${}^\star$S\xspace}
 \nc{\FESSs}{FE${}^\star$Ss\xspace}

 \nc{\DSh}{{\mathsf{DSh}}}
 \nc{\ShC}{{\mathsf{ShC}}}
 \nc{\MZV}{{\mathsf{MZV}}}
 \nc{\FMZV}{{\mathsf{FMZV}}}
 \nc{\FMZsV}{{\mathsf{FMZ}^\star\mathsf{V}}}
 \nc{\MPV}{{\mathsf{MPV}}}
 \nc{\ES}{{\mathsf{ES}}}
 \nc{\FES}{{\mathsf{FES}}}
\nc{\qMZV}{\mathsf{qMZV}}
\nc{\grqMZV}{\mathsf{grqMZV}}

 \nc{\wtcalM}{{\widetilde\calM}}
 \nc{\suf}{{\ast\,}}
 \nc{\sufq}{{\ast_q\,}}
 \nc{\gam}{{\gamma}}
 \nc{\gG}{{\Gamma}}
 \nc{\om}{{\omega}}
 \nc{\vep}{{\varepsilon}}
 \nc{\ga}{{\alpha}}
 \nc{\gl}{{\lambda}}
 \nc{\gb}{{\beta}}
 \nc{\gd}{{\delta}}
 \nc{\orgd}{{\vec \gd\,}}
 \nc{\gs}{{\sigma}}
 \nc{\gth}{{\theta}}
 \nc{\gS}{{\Sigma}}

 \nc{\gk}{{\kappa}}
  \nc{\gz}{{\zeta}}
 \nc{\tgz}{{\tilde{\zeta}}}
 \nc{\gO}{{\Omega}}
 \nc{\sif}{{\mathcal S}}
 \nc{\gt}{{\tau}}
 \nc{\Lra}{\Longrightarrow}
 \nc{\lra}{\longrightarrow}
 \nc{\lmaps}{\longmapsto}
 \nc{\fS}{{\mathsf S}}
 \nc{\DD}{{\mathfrak D}}
 \nc{\Llra}{\Longleftrightarrow}
 \nc{\ol}{\overline}
 \nc{\ola}{\overleftarrow}
 \nc{\lms}{\longmapsto}
 \nc{\zq}{{\zeta_q}}
 \nc{\us}{\underset}
 \nc{\tn}{{\tilde{n}}}
 \nc{\gD}{{\Delta}}
 \nc{\bi}{{\bf i}}
 \nc{\bfone}{{\bf 1}}
 \nc{\bfzero}{{\bf 0}}

 \nc{\bfa}{{\bf a}}
 \nc{\bfb}{{\bf b}}
 \nc{\bfc}{{\bf c}}
 \nc{\bfd}{{\bf d}}
 \nc{\bfe}{{\bf e}}
 \nc{\bff}{{\bf f}}
 \nc{\bfg}{{\bf g}}
 \nc{\bfi}{{\bf i}}
 \nc{\bfj}{{\bf j}}
\nc{\obfi}{{\overrightarrow{\boldsymbol \imath}}}
\nc{\obfj}{{\overrightarrow{\boldsymbol \jmath}}}
\nc{\obfk}{{\overrightarrow{\bf k}}}
\nc{\veps}{{\varepsilon}}
 \nc{\bfn}{{\bf n}}
 \nc{\bfl}{{\bf l}}
 \nc{\bfk}{{\bf k}}
 \nc{\bfm}{{\bf m}}
 \nc{\bfo}{{\bf o}}
 \nc{\bfp}{{\bf p}}
 \nc{\bfq}{{\bf q}}
 \nc{\bfr}{{\bf r}}
 \nc{\bfs}{{\bf s}}
 \nc{\bft}{{\bf t}}
 \nc{\bfu}{{\bf u}}
 \nc{\bfv}{{\bf v}}
 \nc{\bfw}{{\bf w}}
 \nc{\bfx}{{\bf x}}
 \nc{\bfy}{{\bf y}}
 \nc{\bfz}{{\bf z}}
 \nc{\bfA}{{\bf A}}
 \nc{\bfB}{{\bf B}}
 \nc{\bfC}{{\bf C}}
 \nc{\bfD}{{\bf D}}
 \nc{\bfE}{{\bf E}}
 \nc{\bfF}{{\bf F}}
 \nc{\bfG}{{\bf G}}
 \nc{\bfH}{{\bf H}}
 \nc{\bfI}{{\bf I}}
 \nc{\bfJ}{{\bf J}}
 \nc{\bfK}{{\bf K}}
 \nc{\bfL}{{\bf L}}
 \nc{\bfM}{{\bf M}}
 \nc{\bfN}{{\bf N}}
 \nc{\bfO}{{\bf O}}
 \nc{\bfP}{{\bf P}}
 \nc{\bfQ}{{\bf Q}}
 \nc{\bfR}{{\bf R}}
 \nc{\bfS}{{\bf S}}
 \nc{\bfT}{{\bf T}}
 \nc{\bfU}{{\bf U}}
 \nc{\bfV}{{\bf V}}
 \nc{\bfW}{{\bf W}}
 \nc{\bfX}{{\bf X}}
 \nc{\bfY}{{\bf Y}}
 \nc{\bfZ}{{\bf Z}}
 \nc{\bbA}{{\mathbb A}}
 \nc{\bbB}{{\mathbb B}}
 \nc{\bbC}{{\mathbb C}}
 \nc{\bbD}{{\mathbb D}}
 \nc{\bbE}{{\mathbb E}}
 \nc{\bbF}{{\mathbb F}}
 \nc{\bbG}{{\mathbb G}}
 \nc{\bbH}{{\mathbb H}}
 \nc{\bbI}{{\mathbb I}}
 \nc{\bbJ}{{\mathbb J}}
 \nc{\bbK}{{\mathbb K}}
 \nc{\bbL}{{\mathbb L}}
 \nc{\bbM}{{\mathbb M}}
 \nc{\bbN}{{\mathbb N}}
 \nc{\bbO}{{\mathbb O}}
 \nc{\bbP}{{\mathbb P}}
 \nc{\bbQ}{{\mathbb Q}}
 \nc{\bbR}{{\mathbb R}}
 \nc{\bbS}{{\mathbb S}}
 \nc{\bbT}{{\mathbb T}}
 \nc{\bbU}{{\mathbb U}}
 \nc{\bbV}{{\mathbb V}}
 \nc{\bbW}{{\mathbb W}}
 \nc{\bbX}{{\mathbb X}}
 \nc{\bbY}{{\mathbb Y}}
 \nc{\bbZ}{{\mathbb Z}}
 \nc{\bba}{{\mathbb a}}
 \nc{\bbb}{{\mathbb b}}
 \nc{\bbc}{{\mathbb c}}
 \nc{\bbd}{{\mathbb d}}
 \nc{\bbe}{{\mathbb e}}
 \nc{\bbf}{{\mathbb f}}
 \nc{\bbg}{{\mathbb g}}
 \nc{\bbh}{{\mathbb h}}
 \nc{\bbi}{{\mathbb i}}
 \nc{\bbk}{{\mathbb k}}
 \nc{\bbl}{{\mathbb l}}
 \nc{\bbm}{{\mathbb m}}
 \nc{\bbn}{{\mathbb n}}
 \nc{\bbo}{{\mathbb o}}
 \nc{\bbp}{{\mathbb p}}
 \nc{\bbq}{{\mathbb q}}
 \nc{\bbr}{{\mathbb r}}
 \nc{\bbs}{{\mathbb s}}
 \nc{\bbt}{{\mathbb t}}
 \nc{\bbu}{{\mathbb u}}
 \nc{\bbv}{{\mathbb v}}
 \nc{\bbw}{{\mathbb w}}
 \nc{\bbx}{{\mathbb x}}
 \nc{\bby}{{\mathbb y}}
 \nc{\bbz}{{\mathbb z}}

 \nc{\calA}{{\mathcal A}}
 \nc{\calB}{{\mathcal B}}
 \nc{\calC}{{\mathcal C}}
 \nc{\calD}{{\mathcal D}}
 \nc{\calE}{{\mathcal E}}
 \nc{\calF}{{\mathcal F}}
 \nc{\calG}{{\mathcal G}}
 \nc{\calH}{{\mathcal H}}
 \nc{\calI}{{\mathcal I}}
 \nc{\calJ}{{\mathcal J}}
 \nc{\calK}{{\mathcal K}}
 \nc{\calL}{{\mathcal L}}
 \nc{\calM}{{\mathcal M}}
 \nc{\calN}{{\mathcal N}}
 \nc{\calO}{{\mathcal O}}
 \nc{\calP}{{\mathcal P}}
 \nc{\calQ}{{\mathcal Q}}
 \nc{\calR}{{\mathcal R}}
 \nc{\calS}{{\mathcal S}}
 \nc{\calT}{{\mathcal T}}
 \nc{\calU}{{\mathcal U}}
 \nc{\calV}{{\mathcal V}}
 \nc{\calW}{{\mathcal W}}
 \nc{\calX}{{\mathcal X}}
 \nc{\calY}{{\mathcal Y}}
 \nc{\calZ}{{\mathcal Z}}
  \nc{\cala}{{\mathcal a}}
 \nc{\calb}{{\mathcal b}}
 \nc{\calc}{{\mathcal c}}
 \nc{\cald}{{\mathcal d}}
 \nc{\cale}{{\mathcal e}}
 \nc{\calf}{{\mathcal f}}
 \nc{\calg}{{\mathcal g}}
 \nc{\calh}{{\mathcal h}}
 \nc{\cali}{{\mathcal i}}
 \nc{\calj}{{\mathcal j}}
 \nc{\calk}{{\mathcal k}}
 \nc{\call}{{\mathcal l}}
 \nc{\calm}{{\mathcal m}}
 \nc{\caln}{{\mathcal n}}
 \nc{\calo}{{\mathcal o}}
 \nc{\calp}{{\mathsf p}}
 \nc{\calq}{{\mathcal q}}
 \nc{\calr}{{\mathcal r}}
 \nc{\cals}{{\mathcal s}}
 \nc{\calt}{{\mathcal t}}
 \nc{\calu}{{\mathcal u}}
 \nc{\calv}{{\mathcal v}}
 \nc{\calw}{{\mathcal w}}
 \nc{\calx}{{\mathcal x}}
 \nc{\caly}{{\mathcal y}}
 \nc{\calz}{{\mathcal z}}

 \nc{\frakA}{{\mathfrak A}}
 \nc{\frakB}{{\mathfrak B}}
 \nc{\frakC}{{\mathfrak C}}
 \nc{\frakD}{{\mathfrak D}}
 \nc{\frakE}{{\mathfrak E}}
 \nc{\frakF}{{\mathfrak F}}
 \nc{\frakG}{{\mathfrak G}}
 \nc{\frakH}{{\mathfrak H}}
 \nc{\frakI}{{\mathfrak I}}
 \nc{\frakJ}{{\mathfrak J}}
 \nc{\frakK}{{\mathfrak K}}
 \nc{\frakL}{{\mathfrak L}}
 \nc{\frakM}{{\mathfrak M}}
 \nc{\frakN}{{\mathfrak N}}
 \nc{\frakO}{{\mathfrak O}}
 \nc{\frakP}{{\mathfrak P}}
 \nc{\frakQ}{{\mathfrak Q}}
 \nc{\frakR}{{\mathfrak R}}
 \nc{\frakS}{{\mathfrak S}}
 \nc{\frakT}{{\mathfrak T}}
 \nc{\frakU}{{\mathfrak U}}
 \nc{\frakV}{{\mathfrak V}}
 \nc{\frakW}{{\mathfrak W}}
 \nc{\frakX}{{\mathfrak X}}
 \nc{\frakY}{{\mathfrak Y}}
 \nc{\frakZ}{{\mathfrak Z}}
 \nc{\fraka}{{\mathfrak a}}
 \nc{\frakb}{{\mathfrak b}}
 \nc{\frakc}{{\mathfrak c}}
 \nc{\frakd}{{\mathfrak d}}
 \nc{\frake}{{\mathfrak e}}
 \nc{\frakf}{{\mathfrak f}}
 \nc{\frakg}{{\mathfrak g}}
 \nc{\frakh}{{\mathfrak h}}
 \nc{\fraki}{{\mathfrak i}}
 \nc{\frakj}{{\mathfrak j}}
 \nc{\frakk}{{\mathfrak k}}
 \nc{\frakl}{{\mathfrak l}}
 \nc{\frakm}{{\mathfrak m}}
 \nc{\frakn}{{\mathfrak n}}
 \nc{\frako}{{\mathfrak o}}
 \nc{\frakp}{{\mathfrak p}}
 \nc{\frakq}{{\mathfrak q}}
 \nc{\frakr}{{\mathfrak r}}
 \nc{\fraks}{{\mathfrak s}}
 \nc{\frakt}{{\mathfrak t}}
 \nc{\fraku}{{\mathfrak u}}
 \nc{\frakv}{{\mathfrak v}}
 \nc{\frakw}{{\mathfrak w}}
 \nc{\frakx}{{\mathfrak x}}
 \nc{\fraky}{{\mathfrak y}}
 \nc{\frakz}{{\mathfrak z}}
 \nc{\so}{{\mathfrak so}}
 \nc{\slfour}{{\mathfrak sl}_4}
 \nc{\one}{{\bf 1}}
 \nc{\zero}{{\bf 0}}
 \nc{\Qab}{\Q\langle a,b\rangle}

\nc{\sa}{{\mathcyr{sh}}}
\nc{\sha}{\mathbin{\mathcyr{sh}}}

\nc{\zetas}{{\zeta^\star}}

\nc{\db}{{\mathbb D}}

 \nc{\invdots}{{.\text{\raisebox{0.2ex}{$\cdot$}} \text{\raisebox{0.9ex}{$\cdot$}} }}
 \nc{\Sy}{{\mathcal S}}
 \nc{\inv}{{\rm inv}}
 \nc{\Rac}{{\mathcal R}}
 \nc{\dd}{{\mathfrak d}}

 \nc{\gFF}{{\chi}}
 \nc{\gX}{{\varPsi}}
 \nc{\gXs}{\gX^\star}
 \nc{\cv}{{\rm cv}}
  \nc{\myone}{\bfone} 

 \nc{\QX}{{\Q\langle \bfX\rangle}}
 \nc{\QY}{{\Q\langle \bfY\rangle}}
 \nc{\CX}{{\CC\langle \bfX\rangle}}
 \nc{\CY}{{\CC\langle \bfY\rangle}}
 \nc{\QXX}{{\Q\langle\!\langle \bfX\rangle\!\rangle}}
 \nc{\QYY}{{\Q\langle\!\langle \bfY\rangle\!\rangle}}
 \nc{\CXX}{{\CC\langle\!\langle \bfX\rangle\!\rangle}}
 \nc{\CYY}{{\CC\langle\!\langle \bfY\rangle\!\rangle}}

  \nc{\BTF}{{F}}
  \nc{\gbb}{{\beta}}

  \nc{\bfzt}{{{\boldsymbol \zeta}}}

 \nc{\hone}{{\widehat{1}}}
 \nc{\genf}{\genfrac{[}{]}{0pt}{}}

 \nc{\oll}[1]{\underline{#1}}
 \nc{\hyf}{{\text{-}}}
 \nc{\bt}{{\bf 2}}
 \nc{\wbfp}{{\widetilde{\bfp}}}
 \nc{\wbfw}{{\widetilde{\bfw}}}
 \nc{\wdt}{{\widetilde{t}}}
 \nc{\wdl}{{\widetilde{\gl}}}
 \nc{\wdp}{{\widetilde{p}}}
 \nc{\wwbfw}{{\overset{\text{\raisebox{-2pt}{$\approx$}}}{\bfw}}}
 \nc{\wwbfp}{{\overset{\text{\raisebox{-2pt}{$\approx$}}}{\bfp}}}
 \nc{\wwdp}{{\overset{\text{\raisebox{-2pt}{$\approx$}}}{\!p}}}
 \nc{\wwdl}{{\overset{\text{\raisebox{-2pt}{$\approx$}}}{\!\gl}}}
 \nc{\tb}{{\tilde{b}}}
 \nc{\tB}{{\widetilde{B}}}
 \nc{\tX}{{\widetilde{X}}}
 \nc{\tY}{{\widetilde{Y}}}
 \nc{\tbfs}{{\tilde{\bfs}}}
 \nc{\tbft}{{\tilde{\bft}}}
 \nc{\tbfu}{{\tilde{\bfu}}}
 \nc{\ttbfs}{{\hat{\bfs}}}
 \nc{\ttbft}{{\hat{\bft}}}
 \nc{\ttbfu}{{\hat{\bfu}}}
 \nc{\rrho}{{\hat{\rho}}}
 \nc{\ggk}{{\hat{\gk}}}
 \nc{\ggs}{{\hat{\gs}}}
 \nc{\oI}{{\overline{I}}}
 \nc{\bI}{{\bar{I}}}
 \nc{\bJ}{{\bar{J}}}
 \nc{\bK}{{\bar{K}}}

 \nc{\bfgb}{{\boldsymbol \gb}}
 \nc{\bfgl}{{\boldsymbol \gl}}
 \nc{\wbfgl}{{\widetilde{\bfgl}}}
 \nc{\wwbfgl}{{\overset{\text{\raisebox{-2pt}{$\approx$}}}{\bfgl}}}
 \nc{\bfga}{{\boldsymbol \ga}}
 \nc{\bfrho}{{\boldsymbol \rho}}
 \nc{\bfchi}{{\boldsymbol \chi}}

 \nc{\qup}{{q\to 1}}

 \nc{\qbin}{\genfrac{[}{]}{0pt}{}}
 \nc{\wh}{\widehat}
 \nc{\I}{{{\rm I}}\xspace}
 \nc{\tI}{{{{{\rm {\tilde I}}}}\xspace}}
 \nc{\II}{{{\rm I\!I}}\xspace}
 \nc{\III}{{{\rm I\!I\!I}}\xspace}
 \nc{\IV}{{{\rm I\!V}}\xspace}
 \nc{\VV}{{{\rm V}}\xspace}
 \nc{\GG}{{{\rm G}}\xspace}
 \nc{\tIV}{{{{\widetilde{\rm I\!V}}}\xspace}}
 \renewcommand{\O}{{\rm O}}
 \nc{\idf}{{\rm \bfI}}
 \nc{\tfA}{{\tilde{\fA}}}
 \nc{\DS}{{\mathsf{DS}}}
 \nc{\DU}{{\mathsf{DU}}}
 \nc{\sfW}{{\mathsf{W}}}
\nc{\sfZ}{{\mathsf{Z}}}

 \nc{\dR}{{\rm dR}}

\def\myqshu{\joinrel{\!\scriptstyle\amalg\hskip -3.1pt\amalg}\,\hskip -8.5pt\hbox{-}\hskip 5pt}

\begin{document}
\title[Uniform Approach to DBSF and Dualities of Various {\lowercase {$q$}}-MZVs]
{Uniform Approach to Double Shuffle and Duality Relations of Various
{\lowercase {$q$}}-Analogs of Multiple Zeta Values via Rota-Baxter Algebras}

\date{}

\author{Jianqiang Zhao}

\begin{abstract}
The multiple zeta values (MZVs) have been studied extensively in recent years.
Currently there exist a few different types of $q$-analogs of the MZVs
($q$-MZVs) defined and studied by mathematicians and physicists.
In this paper, we give a uniform treatment of these $q$-MZVs
by considering their double shuffle relations (DBSFs) and duality relations.
The main idea is a modification and generalization of the one used by Castillo Medina et al.
who have considered the DBSFs of a special type of $q$-MZVs.
We generalize their method to a few other types of $q$-MZVs including the one defined
by the author in 2003. With different approach, Takeyama has already studied
this type by ``regularization'' and observed that there exist
$\Q$-linear relations which are not consequences of the DBSFs. He also discovered
a new family of relations which we call the duality relations in this paper.
This deficiency of DBSFs occurs among other types, too, so
we generalize the duality relations to all of these values and find that
there are still some missing relations.
This leads to the most general type of $q$-MZVs together with a new kind
of relations called $\bfP$-$\bfR$ relations which are used to lower the deficiencies further.
As an application, we will confirm a conjecture of Okounkov on the dimensions of certain $q$-MZV spaces,
either theoretically or numerically, for the weight up to 12.
Some relevant numerical data are provided at the end.
\end{abstract}

\maketitle

\section{Introduction}
The multiple zeta values are iterated generalizations of the Riemann zeta values
to the multiple variable setting. Euler \cite{Euler1775} first studied the double zeta values
systematically in the 18th century. Hoffman \cite{Hoffman1992} and Zagier \cite{Zagier1994}
independently considered the following more general form in the early 1990's. Let $\N$ be
the set of positive integers.
For any $d\in\N$ and $\bfs=(s_1,\dots,s_d)\in\N^d$ with $s_1\ge 2$ one defines
the \emph{multiple zeta values} (MZVs) as the $d$-fold sum
\begin{equation*}
\zeta(\bfs)=\sum_{k_1>\dots>k_d>0} \frac1{k_1^{s_1}\cdots k_d^{s_d}}.
\end{equation*}
A lot of important and sometimes surprising applications of MZVs have been found in
many areas in mathematics and theoretical physics in recent years, see \cite{Broadhurst1996,Brown2012,GoncharovMa2004,KurokawaLaOc2009,LeMu1995}. One of the
most powerful ideas is to consider the so-called double shuffle
relations (DBSFs). The stuffle relations are obtained directly by
using the above series definition when multiplying two MZVs. The other, the shuffle
relations, can be produced by multiplying their integral representations and using
Chen's theory of iterated integrals \cite{Chen1971}. The interested reader is
referred to the seminal paper \cite{IKZ2006} for more details.

Lagging behind the above development for about a decade,
a few $q$-analogs were proposed and studied
by different mathematicians and physicists. All of these $q$-analogs enjoy
the property that when $q\to 1$ one can recover the ordinary MZVs
defined in the above if no divergence occurs. In this paper, by modifying and generalizing
an idea in \cite{CEM2013} we give a uniform treatment
of these $q$-analogs by using some suitable Rota-Baxter algebras which reflect the
properties of the Jackson's integral representations of these $q$-analogs.

Recall that for any fixed complex number $q$ with $|q|<1$ one can define the $q$-analog of
positive integers by setting $[k]=[k]_q:=1+q+\dots+q^{k-1}=(1-q^k)/(1-q)$ for all $k\in\N$.
To summarize the various versions of $q$-analog of MZVs ($q$-MZVs for abbreviation),
we first define a general type of $q$-MZV of $2d$ variables $s_1,\dots, s_d,
t_1,\dots, t_d\in \Z$
\begin{equation}\label{equ:auxiliaryDefn}
\zeta_q^{\bft}[\bfs]
:=\sum_{k_1>\dots>k_d>0} \frac{q^{k_1t_1+\cdots+k_dt_d}}{[k_1]^{s_1}\cdots[k_d]^{s_d}}
=(1-q)^{|\bfs|}\sum_{ k_1>\dots>k_d>0} \frac{q^{k_1t_1+\cdots+k_dt_d}}{(1-q^{k_1})^{s_1}\cdots (1-q^{k_d})^{s_d}},
\end{equation}
where $|\bfs|=s_1+\dots+s_d$ is called the \emph{weight} and $d$ the \emph{depth}.
The variables of $\bft$ are called \emph{auxiliary variables}.
Also, it is often convenient to study its modified form by dropping the power of $1-q$:
\begin{equation*}
    \frakz_q^{\bft}[\bfs]:=\sum_{ k_1>\dots>k_d>0} \frac{q^{k_1t_1+\cdots+k_dt_d}}{(1-q^{k_1})^{s_1}\cdots (1-q^{k_d})^{s_d}},
\end{equation*}

In the following table, we list a few different versions of $q$-MZVs that
have been studied so far by different authors, except for one new type (type \IV in the table).
We only write down their
modified form although sometimes the original authors only considered $\zeta_q$.
\begin{figure}[!h]
{
\begin{center}
\begin{tabular}{  |c|c|c|c|c| } \hline
Type & Year & Authors & $q$-MZV & DBSF \\ \hline
\   &2001& Schlesinger \cite{Schlesinger2001} & $\frakz_q^{(0,\dots,0)}[s_1,\dots, s_d]$
                    & see \eqref{equ:Schlesinger} \\ \hline
\I  &2002& Kaneko, Kurokawa \& Wakayama \cite{KanekoKuWa2003} & $\frakz_q^{(s-1)}[s]$ (depth=1) & N/A \\ \hline
\I  &2003& Zhao \cite{Zhao2007c} & $\frakz_q^{(s_1-1,\dots, s_d-1)}[s_1,\dots, s_d]$
            & \cite{Takeyama2013}, $\star$ \\ \hline
\II &2003& Zudilin~\cite{Zudilin2003}& $\frakz_q^{(s_1,\dots, s_d)}[s_1,\dots, s_d]$ &$\star$ \\ \hline
\III&2012& Ohno, Okuda \& Zudilin \cite{OOZ2012} & $\frakz_q^{(1,0,\dots, 0)}[s_1,\dots, s_d]$ &
        \cite{CEM2013}, $\star$\\ \hline
\IV &2014& Zhao $\star$ & $\frakz_q^{(s_1-1,s_2,\dots, s_d)}[s_1,\dots, s_d]$ &  $\star$\\ \hline
BK  &2013&Bachmann \& K\"uhn \cite{BachmannKu2013}&$\frakz_q^{\rm BK}[s_1,\dots, s_d]$
        &\cite{Zudilin2014}  \\  \hline
O   &2014& Okounkov \cite{Okounkov2014} & $\frakz_q^\O[s_1,\dots, s_d], s_j\ge 2$ & $\star$ \\ \hline
G   &2003& Zhao \cite{Zhao2007c} & $\frakz_q^{(t_1,\dots, t_d)}[s_1,\dots, s_d]$ & $\star$ \\ \hline
\end{tabular}
\end{center}
}
\caption{A time line of different versions of $q$-MZVs. $\star$=this paper.}
\label{Table:qMZV}
\end{figure}
We notice that in 2004, Bradley \cite{Bradley2005} apparently defined
$\zeta_q^{(s_1-1,\dots, s_d-1)}[s_1,\dots, s_d]$
independently, and later, Okuda and Takeyama also studied some of the relations
among this type of $q$-MZVs in \cite{OkudaTa2007}.
Additionally, it is not hard to see that Schlesinger's version diverges when $|q|<1$ but
can converge if $|q|>1$. In fact, for $\bfs\in\Z^d$
\begin{equation}\label{equ:Schlesinger}
    \frakz_{q^{-1}}^{(0,\dots,0)}[s_1,\dots, s_d]
    =(-1)^{s_1+\dots+s_d} \frakz_q^{(s_1,\dots, s_d)}[s_1,\dots, s_d]
    =(-1)^{s_1+\dots+s_d} \frakz_q^\II[s_1,\dots, s_d].
\end{equation}
So it suffices to consider type \II in order to understand Schlesinger's $q$-MZVs.
The last column of Table~\ref{Table:qMZV} provides the references where
DBSFs are considered systematically (not only the stuffle).

In this paper, we will use suitable Rota-Baxter algebras to study the first four types of $q$-MZVs
listed in Table~\ref{Table:qMZV} in details. We also briefly consider the general type G and
Okounkov's type O $q$-MZVs. Note that the numerators inside the summands of
$\zeta_q^{\rm BK}$ and $\zeta_q^\O$
are not exact powers of $q$, but some polynomials of $q$ enjoying nice properties.
Further, for $\zeta_q^\O$ the polynomial numerator is at worst a sum of two $q$-powers so
our method can still work. See Corollary~\ref{cor:qIteratedZetaOk}.
It may be difficult to use the approach here to study the Bachmann and K\"uhn's type since
the numerators are much more complicated.

In the classical setting, the so-called regularized DBSFs play extremely
important roles in discovering and proving $\Q$-linear relations among MZVs.
The first serious attempt to discover the DBSFs among $q$-MZVs
was carried out by the author in \cite{Zhao2007c} by using Jackson's $q$-integrals.
However, the computation there was too complicated so only very few relations
were found successfully. The real breakthrough came with Takeyama's successful
application of Hoffman's algebras to study type \I $q$-MZVs in \cite{Takeyama2013}.
However, his approach to the shuffle relations relies on some auxiliary
multiple polylogarithm functions and consequently it is very hard to see why
these relations should hold.

The situation looks much better with the appearance of a recent paper \cite{CEM2013}
by Castillo Medina et al.\ who generalized Chen's iterated integrals to
iterated Jackson's $q$-integrals to
study type \III $q$-MZVs by using Rota-Baxter algebra techniques.
Motivated by this new idea, in this paper we will consider all the $q$-MZVs
of type \I, \II, \III and \IV by finding/using their correct realizations in terms of
iterated Jackson's $q$-integrals. Then by combining the Rota-Baxter algebra technique
and Hoffman's algebra of words we are able to study the DBSFs of all of these $q$-MZVs.

When one considers the $\Q$-linear relations among the ordinary MZVs,
the main difficulty lies in the insufficiency of DBSFs produced using
only admissible arguments. In the $q$-analog setting, the situation is
only partially similar and is sometimes much more complicated.

For type \I $q$-MZVs, our computation shows that the DBSFs CAN provide
all the $\Q$-linear relations. However, in order to study these relations,
as Takeyama noticed first, one has to
enlarge the set of type \I $q$-MZVs to something we call type $\tI$ $q$-MZVs which
are a kind of ``regularized'' $q$-MZVs in the sense that
one needs to consider some convergent versions of $q$-MZVs when $s_1=1$
by modifying the auxiliary variables of $\bft$. But for these type $\tI$ $q$-MZVs themselves,
DBSFs are insufficient to provide all the $\Q$-linear relations and a certain ``Resummation Identity"
defined by Takeyama is required. In this paper, we will adopt the term ``duality''
due to its similarity to the duality relations of the ordinary MZVs. Moreover,
for type $\tI$ $q$-MZVs of weight bounded by $w$ there are often still missing
relations even after we consider both DBSFs and duality relations within
same weight and depth range. These missing relations can be
recovered only after we increase the weight and depth. This phenomenon is not
unique to type $\tI$ $q$-MZVs. We have recorded this fact by using the ``deficiency''
numbers listed in the tables in the last section of this paper.

Similar to type \I, we find that type \IV $q$-MZVs also need to be ``regularized'' when $s_1=1$.
Again, we achieve this by introducing some convergent versions of the $q$-MZVs by modifying
the auxiliary variables in $\bft$.

It turns out that type \II $q$-MZVs behave most regularly and enjoy some properties closest
to those of the ordinary MZVs. For example, their duality relations (see Theorem~\ref{thm:resummation-dualityII}) have the cleanest form. Moreover, every other type
of $q$-MZVs considered in this paper can be converted to type \II. But still, there are
relations that cannot be proved by DBSFs and dualities, at least when one is confined
within the same weight and depth range. In fact, we find three independent $\Q$-linear
relations in weight 4 that can only be proved when we consider weight 5 DBSFS and dualities.

All type \III $q$-MZVs are convergent, even for negative arguments. For simplicity,
in this paper we only consider those nonnegative arguments $s_1,\dots,s_d$ with $s_1\ge 1$.
In this case, the DBSFs are still insufficient. In the last section, we will see that in
weight 3 there is already a missing
relation which can be recovered by the duality. Essentially because of the need to
apply the duality relations we have to modify the original Jackson integral representation given
in \cite{CEM2013}. See the remarks after Theorem~\ref{thm:qIteratedZetaIIICEM}.
In contrast to the other types of $q$-MZVs, we cannot suppress the deficiency for type \III even
if we consider more DBSFs and duality relations by increasing the weight and depth. This
might be caused by our restriction of only nonnegative arguments and
thus further investigations are called for.

On the other hand, we can improve the above situation by considering the
more general type G values. All the missing relations are thus proved up to
and including weight 4 and at same time both deficiencies
are decreased in weight 5 and 6. The key idea here is to convert all type G values
to type II values by using a new kind of relations called $\bfP$-$\bfR$ relations.

We point out that our method can be easily adapted to study $q$-MZVs of
the following general forms:
\begin{equation*}
\frakz_q^{(s_1-a_1,\dots, s_d-a_d)}[s_1,\dots, s_d], \qquad
\frakz_q^{(a_1,\dots, a_d)}[s_1,\dots, s_d],
\end{equation*}
where $a_1\ge a_2 \ge \cdots\ge a_d\ge 0$ are all fixed integers. Furthermore, when the weight is
not too large, our method can be programmed to compute all the relations among
$q$-MZVs of the general form $\frakz_q^{\bft}[\bfs]$ when $\bft$ is taken within
a certain range. This will be carried out in section~\ref{sec:TypeG}.

As an application, for small weight cases
it is possible to confirm Okounkov's conjecture \cite{Okounkov2014}
on the dimension of the $q$-MZVs $\frakz_q^\O[\bfs]$ using
Corollary~\ref{cor:qIteratedZetaOk}. We do this numerically up to weight 12 and give
rigorous proof up to weight 6 (both inclusive).

Throughout the paper we will use the modified form $\frakz_q$. All the results can be
translated into the standard form $\zeta_q$ by inserting the correct powers of $(1-q)^w$,
where $w$ is the corresponding weight, into the formulas.

\section{Convergence domain for $q$-MZVs}
We need the following result to find the convergence domain for different types of $q$-MZVs.
It is Proposition~2.2 of \cite{Zhao2007c} where the order of the indexes in the definition
of $\zeta_q^{(t_1,\dots, t_d)}[s_1,\dots, s_d]$
(denoted by $f_q(s_d,\dots, s_1;t_d,\dots, t_1)$ in loc.\ cit.) is opposite to this paper.

\begin{prop}\label{prop:fAuxFunction}
The function $\zeta_q^{(t_1,\dots, t_d)}[s_1,\dots, s_d]$
converges if $\Re(t_1+\dots+t_j)>0$ for all $j=1,\dots,d$. It can
be analytically continued to a meromorphic function over $\CC^{2d}$
via the series expansion
\begin{equation}\label{equ:merom}
\zeta_q^{(t_1,\dots, t_d)}[s_1,\dots, s_d]=(1-q)^{|\bfs|}
\sum_{r_1,\dots,r_d=0}^{\infty} \prod_{j=1}^d
\left[{s_j+r_j-1\choose r_j}
\frac{q^{(d+1-j)(r_j+t_j)}}{1-q^{r_1+t_1+\cdots+r_j+t_j}} \right].
\end{equation}
It has the following (simple)
poles: $t_1+\dots+t_j\in \Z_{\le 0}+\frac{2\pi i}{\log q}\Z$
for $j=1,\dots,d$.
\end{prop}

\begin{cor}\label{cor:convergenceDomain}
Let $\bfs=(s_1,\dots, s_d)\in\Z^d$. Then
\begin{itemize}
  \item [\upshape{(i)}]$\zeta_q^\I[\bfs]$ converges if  $s_1+\dots+s_j>j$ for all $j=1,\dots,d$.

  \item [\upshape{(ii)}]$\zeta_q^\II[\bfs]$ converges if  $s_1+\dots+s_j>0$ for all $j=1,\dots,d$.

  \item [\upshape{(iii)}]$\zeta_q^\III[\bfs]$ always converges.

  \item [\upshape{(iv)}]$\zeta_q^\IV[\bfs]$ converges if  $s_1+\dots+s_j>1$ for all $j=1,\dots,d$.
\end{itemize}
\end{cor}

\begin{defn}\label{defn:admissibleComposition}
For convenience, a composition $\bfs\in\Z_{\ge 0}$ is said to be \emph{type $\gt$-admissible}
if $\bfs$ satisfies the condition for type $\gt$ $q$-MZVs in the corollary.
Here and in what follows, $\gt=\I$, \II, \III, or \IV.
\end{defn}

\section{Rota-Baxter algebra}\label{sec:Rota-BaxterAlg}
In this section we briefly review some fundamental facts
of Rota-Baxter algebras which will be crucial in the study of
the $q$-analog of \emph{shuffle} relations for all of $q$-MZVs
considered in this paer.

\begin{defn}\label{defn:Rota-Baxter algebra}
Fix an algebra $A$ over a commutative ring $R$ and an element $\gl\in R$.
We call $A$ a Rota-Baxter $R$-algebra and $\calP$ a Rota-Baxter operator
of weight $\gl$ if the operator $\calP$ satisfies the following Rota-Baxter
relation of weight $\gl$:
\begin{equation}\label{equ:Rota-BaxterRelation}
\calP(x)\calP(y)=\calP(\calP(x)y)+\calP(x\calP(y))+\gl\calP(xy) \quad\forall x,y\in A.
\end{equation}
\end{defn}

Recall that for any continuous function $f(x)$ on $[\ga,\gb]$
Jackson's $q$-integral is defined by
\begin{equation}\label{equ:q-intJackson}
\int_\ga^\gb f(x)\,d_qx:=\sum_{k\ge 0} f\big(\ga+q^k(\gb-\ga)\big)
(q^k-q^{k+1})(\gb-\ga).
\end{equation}
Taking $\ga=0$ and $\gb=t$ in \eqref{equ:q-intJackson} we now set
\begin{equation}\label{equ:Jackson}
\bfJ[f](t):= (1-q)\sum_{k\ge 0} f(q^kt) q^kt
=(1-q)\sum_{k\ge 0}\bfE^k\big[\idf\cdot f\big](t)
= (1-q)\bfP [\idf\cdot f](t). 	
\end{equation}
where $\idf(t)=t$ is the identity function,
\begin{equation*}
\bfE[f](t):=\bfE_q[f](t):=f(qt),\ \text{and}\
\bfP[f](t):=\bfP_q[f](t):=f(t)+f(qt)+f(q^2t)+\cdots
\end{equation*}
are the $q$-\emph{expanding} and the (\emph{principle}) $q$-\emph{summation operators}, respectively.
We also need to define the (\emph{remainder}) $q$-\emph{summation operator}
\begin{equation*}
\bfR[f](t):=\bfR_q[f](t):=f(qt)+f(q^2t)+\cdots=(\bfP[f]-[f])(t).
\end{equation*}
So, $\bfP$ is the principle part (i.e. the whole thing) while $\bfR$ is the
remainder (i.e., without the first term). Clearly, $\bfP=\bfR+\idf$ where,
as an operator, $\idf[f]=f$. This implies that $\bfP\bfR=\bfR\bfP$.

Let $t\Q[\![t,q]\!]$ be the ring of formal series in two variables with $t>0$.
Then $\bfJ$, $\bfE$, $\bfP$ and $\bfR$ are all $\Q[[q]]$-linear endomorphism of $t\Q[\![t,q]\!]$.
We can further define the inverse to
$\bfP$ which is called the $q$-\emph{difference operator}:
\begin{equation}
\label{equ:qDifference}
	\bfD :=\idf -\bfE.
\end{equation}

The following results extend those of \cite[(21)-(23)]{CEM2013}. In the final computation
we will not need $\bfD$ since we will only consider nonnegative arguments
in all the $q$-MZVs. But in the theoretical part of this paper we do need to use $\bfD$
for type \III $q$-MZVs.

\begin{prop}\label{prop:checkRBalg}
For any $f,g\in t\Q[\![t,q]\!]$ we have
{\allowdisplaybreaks
\begin{align}\label{equ:RisRBalg}
	\bfP[f]\bfP[g] &\, =\bfP\big[\bfP[f]g\big]+\bfP\big[f\bfP[g]\big] -\bfP[fg],\\
\label{equ:PisRBalg}
	\bfR[f]\bfR[g] &\, =\bfR\big[\bfR[f]g\big]+\bfR\big[f\bfR[g]\big]+\bfR[fg],\\
\label{equ:PRproduct}
	\bfR[f]\bfP[g] &\, =\bfR\big[\bfR[f]g\big]+\bfR\big[f\bfR[g]\big]+\bfR[f]g+\bfR[fg],\\
\label{equ:JisRBalg}
 \bfJ[f]\bfJ[g] &\, =\bfJ\big[\bfJ[f]g\big]+\bfJ\big[f\bfJ[g]\big] - (1-q)\bfJ\big[\idf f g\big],\\
\label{equ:JisRBalg2}
 &\, =\bfJ\big[ f\bfJ[g] \big]+ q\bfJ\big[ \bfJ \big[E[f] \big]g \big],\\
\label{equ:DDproduct}
	\bfD[f]\bfD[g] &\, =\bfD[f]g+ f\bfD[g]-\bfD[fg],\\
\label{equ:DPproduct}
	\bfD[f]\bfP[g] &\, =\bfD\big[f\bfP[g]\big]+\bfD[f]g - fg,\\
\label{equ:DRproduct}
	\bfD[f]\bfR[g] &\, =\bfD\big[f\bfR[g]\big]+\bfD[fg] - fg,\\
\label{equ:DPinverse}
	\bfD\bfP &\, =\bfP\bfD=\idf, \qquad \bfP\bfR=\bfR\bfP.
\end{align}}
\end{prop}
\begin{proof}
The identities \eqref{equ:RisRBalg}, \eqref{equ:DDproduct} and
\eqref{equ:DPproduct} are just (21), (23) and (26) of \cite{CEM2013}, respectively.
All the others follow from  $\bfR=\bfP-\idf$ easily.
\end{proof}

By Proposition~\ref{prop:checkRBalg} we see that $\bfP$ and $\bfR$ are both Rota-Baxter
operators on $t\Q[\![t,q]\!]$ (of weight $-1$ and 1, respectively) but $\bfD$ is not.
In fact, $\bfD$ satisfies the condition \eqref{equ:DDproduct} of a differential
Rota-Baxter operator \cite{GuoKe2008}. Moreover, it is \emph{invertible} in the sense that
Rota-Baxter operator $\bfP$ and the differential $\bfD$ are mutually inverse by \eqref{equ:DPinverse}.

We end this section with an identity which will be used to interpret Takeyama's Resummation
Identity in \cite{Takeyama2013}. For any $n\in\N$, set
\begin{equation*}
    \bfP^n=\underbrace{\bfP\circ \cdots \circ\bfP}_{n\text{ times}}\quad\text{and}\quad
    \bfR^n=\underbrace{\bfR\circ \cdots \circ\bfR}_{n\text{ times}}.
\end{equation*}

\begin{thm}\label{thm:ResumIdRBform}
Let $d\in\N$ and $\ga_j,\gb_j\in\N$ for all $j=1,\dots,d$. Let $\bfy(t)=\frac{t}{1-t}$.
Then we have
\begin{equation}\label{equ:dualityRBAlg}
\bfR^{\ga_{1}}\bfy^{\gb_{1}}\cdots\bfR^{\ga_{\ell}}\bfy^{\gb_{\ell}}(t)
=\sum_{\substack{
j_1\ge\gb_1,\dots,\, j_\ell\ge \gb_\ell \\
k_1\ge\ga_1,\dots,\, k_\ell\ge \ga_\ell}}
\prod_{r=1}^\ell
 \bigg[\binom{j_r-1}{\gb_r-1}\binom{k_r-1}{\ga_r-1}  q^{k_r\sum_{s=r}^\ell j_s} t^{j_r}\bigg].
\end{equation}
\end{thm}
\begin{proof}
First we show that
\begin{equation}\label{equ:ResumIdRonly}
\bfR^{\ga}(t^j)=\frac{q^{\ga j } t^j}{(1-q^j)^\ga }
\end{equation}
In deed, if $\ga=1$ then
\begin{equation*}
 \bfR(t^j) =\sum_{k\ge 1} q^{kj} t^j =\frac{q^j t^j}{1-q^j}.
\end{equation*}
So \eqref{equ:ResumIdRonly} can be proved easily by induction.

Now we proceed to prove that for any integer $m\ge 0$
\begin{equation}\label{equ:dualityRBAlgInduction}
\bfR^{\ga_{1}}\bfy^{\gb_{1}}\cdots\bfR^{\ga_{\ell}}\Big(\bfy^{\gb_{\ell}}(t)\cdot t^m\Big)
=\sum_{\substack{
j_1\ge\gb_1,\dots,\, j_\ell\ge \gb_\ell \\
k_1\ge\ga_1,\dots,\, k_\ell\ge \ga_\ell}}
t^m\prod_{r=1}^\ell
 \bigg[\binom{j_r-1}{\gb_r-1}\binom{k_r-1}{\ga_r-1}  q^{k_r(m+\sum_{s=r}^\ell j_s)} t^{j_r}\bigg].
\end{equation}
If $\ell=1$ then we have
\begin{align*}
\bfR^{\ga}\Big(\bfy^{\gb}(t)\cdot t^m \Big)
=&\, \bfR^{\ga} \left(\Big(\frac{t}{1-t}\Big)^\gb t^m\right)
=\bfR^{\ga} \sum_{j\ge 0} \binom{\gb+j-1}{j} t^{m+\gb+j} \\
=&\, \bfR^{\ga} \sum_{j\ge \gb} \binom{j-1}{\gb-1} t^{m+j}  \\
=&\, \sum_{j\ge \gb} \binom{j-1}{\gb-1} \frac{q^{\ga(m+j)} t^{m+j}}{(1-q^{m+j})^\ga }
        \qquad(\text{by \eqref{equ:ResumIdRonly}})  \\
=&\, \sum_{j\ge \gb} \binom{j-1}{\gb-1} \sum_{k\ge 0}  \binom{\ga+k-1}{k} q^{(\ga+k)(m+j)} t^{m+j} \\
=&\, \sum_{j\ge \gb}\sum_{k\ge\ga} \binom{j-1}{\gb-1}\binom{k-1}{\ga-1}  q^{k(m+j)} t^{m+j}.
\end{align*}
This proves \eqref{equ:dualityRBAlgInduction} when $\ell=1$.
In general
\begin{multline*}
\bfR^{\ga_{1}}\bfy^{\gb_{1}}\cdots\bfR^{\ga_{\ell-1}}\Big(\bfy^{\gb_{\ell-1}}(t)
    \Big(\bfR^{\ga_{\ell}}\Big(\bfy^{\gb_{\ell}}(t)\cdot t^m\Big)\Big)\Big) \\
=\sum_{j_\ell\ge \gb_\ell}\sum_{k_\ell\ge\ga_\ell} \binom{j_\ell-1}{\gb_\ell-1}\binom{k_\ell-1}{\ga_\ell-1}  q^{k_\ell(m+j_\ell) }
\bfR^{\ga_{1}}\bfy^{\gb_{1}}\cdots\bfR^{\ga_{\ell-1}}\Big(\bfy^{\gb_{\ell-1}}(t) \cdot t^{m+j_\ell}\Big).
\end{multline*}
So \eqref{equ:dualityRBAlgInduction} follows immediately by induction.
We can now finish the proof of the theorem by taking $m=0$.
\end{proof}

\begin{cor}\label{cor:permuteRy}
Let $d\in\N$ and $\ga_j,\gb_j\in\N$ for all $j=1,\dots,d$. Then we have
\begin{equation}\label{equ:permuteRyLimited}
\bfR^{\ga_{1}}\bfy^{\gb_{1}}\cdots\bfR^{\ga_{\ell}}\bfy^{\gb_{\ell}}(1)
=\bfR^{\gb_{\ell}}\bfy^{\ga_{\ell}}\cdots\bfR^{\gb_{1}}\bfy^{\ga_{1}}(1).
\end{equation}
\end{cor}
\begin{proof}
In \eqref{equ:dualityRBAlg} we use the substitutions $j_r\leftrightarrow k_{\ell+1-r}$
for all $r=1,\dots,\ell$. Then we have
\begin{align*}
\sum_{r=1}^\ell \sum_{s=r}^\ell j_s k_r
\lra &\, \sum_{r=1}^\ell \sum_{s=r}^\ell j_{\ell+1-r} k_{\ell+1-s}
= \sum_{s=1}^\ell \sum_{r=1}^s j_{\ell+1-r} k_{\ell+1-s} \\
=&\, \sum_{s=1}^\ell \sum_{r=1}^{\ell+1-s} j_{\ell+1-r} k_s
=\sum_{s=1}^\ell \sum_{r=s}^{\ell} j_r k_s
=\sum_{r=1}^\ell k_r\sum_{s=r}^{\ell} j_s.
\end{align*}
which follows from $s\leftrightarrow\ell+1-s$ followed by $r\leftrightarrow\ell+1-r$ and $r\leftrightarrow s$.
This proves the corollary.
\end{proof}

\section{{\lowercase {$q$}}-analogs of Hoffman algebras} \label{sec:qHoffmanALg}
We know that (regularized) DBSFs lead to many (and conjecturally all)
$\Q$-linear relations among MZVs. The key idea here was first suggested by Hoffman \cite{Hoffman1997}
who used some suitable algebra of words to codify both the stuffle (also called harmonic
shuffle \cite{Terasoma2006} or quasi-shuffle \cite{Hoffman2000}) relations coming from the series
representation of MZVs and the shuffle relations coming from the iterated integral
expressions of MZVs. The detailed regularization process can be found in \cite{IKZ2006}.
To study similar relations of the $q$-MZVs
we should modify the Hoffman algebras in the $q$-analog setting.

The following definition for type I $q$-MZVs was first proposed by Takeyama~\cite{Takeyama2013}.
We adopt different notations here in hoping to give a uniform and more
transparent presentation for all the four types of $q$-MZVs.

First we consider some algebras which will be used
to define the stuffle relations later.

\begin{defn}\label{defn:q-AlgebrafAgth}
Let $X_\gth^*$ be the set of words
on the alphabet $X_\gth=\{a,a^{-1},b,\gth\}$.
Denote by $\fA_\gth=\Q\langle a,b,\gth\rangle$
the noncommutative polynomial $\Q$-algebra of words from $X_\gth^*$. Set
\begin{equation*}
\gam:=b-\gth,\qquad  z_{s}:=a^{s-1}b,\quad z'_{s}:=a^{s-1}\gth,\quad s\in \Z.
\end{equation*}
Let $Y_\tI:=\{\gth\}\cup\{z_{k}\}_{k\ge 1}$,
$Y_\II:=\{z'_{k}\}_{k\ge 0}$, $Y_\III:=\{z_{k}\}_{k\in\Z}$
and $Y_\tIV:=\{\gth\}\cup\{z'_{k}\}_{k\ge 0}$. We point out that $z_0,z'_0\ne \myone$
where $\myone$ is the empty word.
We put a tilde on top of \I and \IV each since we need to consider some kind of regularization
due to convergence issues involved in type \I and \IV $q$-MZVs. This is realized by
the introduction of the letter $\gth$.
Again, we use $Y_\gt^*$ to denote the set of words generated on $Y_\gt$ for any type $\gt$.

Let $\fA_\tI^1$, $\fA_\II^1$, $\fA_\III$ and $\fA_\tIV$
be the subalgebra of $\fA_\gth$ freely generated by the sets
$Y_\tI$, $Y_\II$, $Y_\III$ and $Y_\tIV$, respectively.
Set
\begin{equation*}
\fA_\III^1=\sum_{k\in\Z} z'_{k}\fA_\III \not\subset \fA_\III,\qquad
\fA_\tIV^1:=\Q\myone+\gth\fA_\tIV+\sum_{k\ge 1}z_{k}\fA_\tIV \not\subset \fA_\tIV.
\end{equation*}
Here, all integer subscripts are allowed in $Y_\III$ because
type \III $q$-MZVs converge for all integer arguments. Further, we define
the following subalgebras corresponding to the convergent values:
\begin{alignat*}{3}
\fA_\I^0:=&\, \Q\myone+\sum_{k\ge 2}z_{k}\fA_\tI^1, \qquad &&
\fA_\tI^0:=\Q\myone+\gth\fA_\tI^1+\sum_{k\ge 2}z_{k}\fA_\tI^1 \subsetneq \fA_\tI^1, \\
\fA_\II^0:=&\, \Q\myone+\sum_{k\ge 1}z'_{k}\fA_\II^1 \subsetneq \fA_\II^1, \qquad &&
\fA_\III^0:=\fA_\III^1, \\
\fA_\IV^0:=&\, \Q\myone+\sum_{k\ge 2}z_{k}\fA_\tIV, \qquad &&
\fA_\tIV^0:=\Q\myone+\gth\fA_\tIV+\sum_{k\ge 2}z_{k}\fA_\tIV \subsetneq \fA_\tIV^1.
\end{alignat*}
For each type $\gt$ the words in $\fA_\gt^0$ are called \emph{type $\gt$-admissible}.
This is consistent with Definition~\ref{defn:admissibleComposition}
since we only consider non-negative compositions  $\bfs$.

To define the stuffle product for type $\gt=\tI$ and \II, similar to the MZV case we
define a commutative product $[-,-]_\gt$ first:
\begin{equation}\label{equ:squareBkRels}
[z_{k},z_{l}]_\tI=z_{k+l}+z_{k+l-1},
\quad
[\gth, z_{k}]_\tI=z_{k+1},
\quad
[\gth,\gth]_\tI=z_{2}-\gth,
\quad
[z'_{k},z'_{l}]_\II=z'_{k+l}
\end{equation}
for all $k, l\ge 1$.
Now we define the stuffle product $*_\gt$ on $\fA_\gt^1$
inductively as follows. For any words $\bfu,\bfv\in\fA_\gt^1$ and
letters $\ga,\gb\in Y_\gt$,
we set $\myone*_\gt\bfu=\bfu=\bfu*_\gt\myone$ and
\begin{equation}\label{equ:qStuffle1WordForm}
 (\ga\bfu)*_\gt(\gb\bfv)=\ga(\bfu*_\gt\gb\bfv)+\gb(\ga\bfu*_\gt\bfv)+[\ga,\gb]_\gt(\bfu*\bfv).
\end{equation}
\end{defn}

\begin{rem}

(i). The definition for $*_\tI$ is the same as in \cite{Takeyama2013}.

(ii).
One can check that $*_\gt$ is well-defined for $\gt=\tI$ and \II. Namely,
$\bfu*_\gt\bfv\in\fA_\gt^1$ if $\bfu,\bfv\in\fA_\gt^1$.

(iii).
It is not hard to check that for $\gt=\tI$ and \II,
$(\fA_\gt^0,*_\gt)\subset(\fA_\gt^1,*_\gt)$ as subalgebras.
\end{rem}

In \cite{CEM2013}, the stuffle product $\myqshu$ for type \III $q$-MZVs is defined.
We will modify this in the following way (see the remarks after Theorem~\ref{thm:qIteratedZetaIIICEM}).
Our modified stuffle product for type \III $q$-MZVs will be denoted by $*_\III$.

\begin{defn}\label{defn:StuffleIII}
Define the injective shifting operator $\sif_{-}$ on any word of $\fA_\III^1$ by
acting on the first letter:
\begin{equation}\label{equ:sifOperator}
	\sif_{-}(z'_n \bfw):=z_n\bfw - z_{n-1}\bfw  \quad \text{ for all $n\in\Z$ and $\bfw\in Y_\III^*$}.
\end{equation}
For any $k,l\in\Z$ and any $\bfu,\bfv\in Y_\III^*$
define the stuffle product $*_\III$ by
\begin{equation*}
z'_k\bfu *_\III z'_l\bfv
=z'_k\big(\bfu*_\II\sif_{-}(z'_l\bfv)\big)
    +z'_l\big(\sif_{-}(z'_k\bfu)*_\II\bfv\big)
    +(z'_{k+l}-z'_{k+l-1})(\bfu*_\II\bfv).
\end{equation*}
Here $*_\II$ is the ordinary stuffle with $[z_r,z_s]_\II=z_{r+s}$ for all $r,s\in Z$.
\end{defn}

For type $\tIV$, we provide a definition similar to type \III.

\begin{defn}\label{defn:StuffleIV}
Define a shifting operator $\sif_+$ similar to \eqref{equ:sifOperator} by
\begin{equation*}
	\sif_{+}(z_n \bfw):=z_n\bfw+z_{n-1}\bfw  \quad \text{ for all $n\in\N$ and $\bfw\in Y_\tIV^*$}.
\end{equation*}
Then, for any $k,l\ge 1$ and any $\bfu,\bfv\in Y_\tIV^*$ we set
\begin{align*}
z_k\bfu *_\tIV z_l\bfv
&\, =z_k\big(\bfu*_\II\sif_{+}(z_l\bfv)\big)
    +z_l\big(\sif_{+}(z_k\bfu)*_\II\bfv\big)
    +(z_{k+l}+z_{k+l-1})(\bfu*_\II\bfv),\\
z_k\bfu *_\tIV \gth \bfv
=\gth \bfv *_\tIV z_k\bfu
&\, =z_k\big(\bfu*_\II \gth \bfv \big)
    +\gth \big( \sif_{+}(z_k\bfu) *_\II \bfv\big)
    +z_{k+1}(\bfu*_\II\bfv),\\
\gth\bfu *_\tIV \gth \bfv
&\, =\gth \big(\bfu*_\II \gth \bfv \big)
    +\gth \big( \gth \bfu *_\II \bfv\big)
    +(z_2-\gth)(\bfu*_\II\bfv),
\end{align*}
where $*_\II$ is the ordinary stuffle with $[\gth,\gth]_\II=z_2$,
$[z_r,\gth]_\II=z_{r+1}$ and $[z_r,z_s]_\II=z_{r+s}$ for all $r,s\ge 1$.
\end{defn}

\begin{lem} \label{lem:StuffleIVwell-define}
The stuffle products $*_\III$ and $*_\tIV$ are both well-defined.
Namely, if $\bfu,\bfv\in\fA_\gt^1$ then
$\bfu *_\tIV\bfv\in\fA_\gt^1$ for $\gt=\III$ or $\tIV$.
\end{lem}
\begin{proof}
We prove the lemma for type $\tIV$ only. Type \III is similar but simpler.

First we notice that $k+l-1\ge 1$ if $k,l\ge 1$. So the first word of each of
the terms of $\bfu*_\tIV\bfv$ has the right form. We need to show that after truncating
the first word each term lies in $\fA_\tIV$. Notice that
$\sif_{+}(z_l\bfv),\sif_{+}(z_k\bfu)\in\fA_\tIV$ and $*_\II$ does not
decrease the size the subscripts (which are all non-negative). The lemma is now proved.
\end{proof}

\begin{prop}\label{prop:stuffleCommAss}
Let $\gt=\tI, \II$, \III or $\tIV$. Then the stuffle algebras $(\fA_\gt^1,*_\gt)$ are
all commutative and associative.
\end{prop}
\begin{proof}
This follows from the fact that the product $[-,-]_\gt$ are all
commutative and associative which can be verified easily.
\end{proof}

We now turn to the shuffle algebra which is an analog of the corresponding algebra for MZVs
reflecting the properties of their representations using iterated integrals.
\begin{defn}\label{defn:q-AlgebrafApi}
Let $X_\pi=\{\pi,\gd,y\}$ be an alphabet
and $X_\pi^*$
be the set of words generated by $X_\pi$.
Define $\fA_\pi=\Q\langle \pi,\gd,y\rangle$ to be the noncommutative
polynomial $\Q$-algebra of words of $X_\pi^*$.
We may embed $\fA_\gth$ defined by Definition~\ref{defn:q-AlgebrafAgth}
as a subalgebra of $\fA_\pi$ in two different ways: put $\rho=\pi-\myone$ and let
\begin{align*}
(A)&\quad  a:=\pi,\quad  a^{-1}:=\gd,\, \quad b:=\pi y,\quad  \gth=\rho y\quad \Longrightarrow \quad   \gam:=y,\\
(B)&\quad  a:=\rho,\quad  a^{-1}:=-,  \quad b:=\pi y,\quad  \gth=\rho y\quad \Longrightarrow \quad   \gam:=y .
\end{align*}
We denote the image of the embedding by $\fA_\gth^{(A)}$ and $\fA_\gth^{(B)}$, respectively.
The dash $-$ for the image of $a^{-1}$ in (B) means it does not matter what image we choose
since $a^{-1}$ only appears when we consider type \III $q$-MZVs using (A). We will use
embedding (B) for the other three types for which $a^{-1}$ will not be utilized essentially
because of convergence issues.
\end{defn}

\section{{\lowercase {$q$}}-stuffle relations}\label{sec:qStuffleVersion}
First we define the $\Q$-linear realization maps
$\frakz_{q}:\fA_\gt^0\to\CC$ ($\gt=\tI,\II$) by $\frakz_{q}[\myone]=1$ and
\begin{equation*}
\frakz_{q} [y^\gt_{1}\dots y^\gt_{d}]
:=\sum_{k_{1}>\cdots>k_d>0} M^\gt_{k_{1}}(y^\gt_1)\dots M^\gt_{k_{d}}(y^\gt_d),
\end{equation*}
where $y^\gt_{1}\dots y^\gt_{d}\in\fA_\gt^0$ and
the $\Q$-linear maps
\begin{equation*}
M^\tI_k(\gth):=\frac{q^k}{(1-q^k)},\quad
M^\tI_k(z_{s}):=\frac{q^{(s-1)k}}{(1-q^k)^s},\quad
M^\II_k(z'_{s}):=\frac{q^{sk}}{(1-q^k)^s}.
\end{equation*}
Note that $M^\tI_k(\gam)=M^\tI_k(z_{1}-\gth)=1.$ For example, we have
\begin{equation*}
    \frakz_q[z_2z_5\gam^2 z_1]= \frakz_{q}^{(1,4,0,0,0)}[2,5,0,0,1], \quad
 \frakz_q[\gth z_7 \gth z_4]= \frakz_{q}^{(1,6,1,3)}[1,7,1,4],
\end{equation*}
which are not $q$-MZVs of type I.

For type $\gt=\III$ or $\tIV$, we similarly define the $\Q$-linear realization maps
$\frakz_q:\fA_\gt^0\to\CC$ by $\frakz_q[\myone]=1$ and
\begin{equation*}
\frakz_q[y^\gt_{1}\dots y^\gt_{d}]
:=\sum_{k_{1}>\cdots>k_d>0} M^{1,\gt}_{k_{1}}(y^\gt_1)M^\gt_{k_{2}}(y^\gt_2)\dots M^\gt_{k_{d}}(y^\gt_d),
\end{equation*}
where $y^\gt_{1}\dots y^\gt_{d}\in\fA_\gt^0$ and
the $\Q$-linear maps
\begin{align*}
M^{1,\III}_k(z'_{s}):=&\, \frac{q^k}{(1-q^k)^s},\quad  M^\III_k(z_{s}):=\frac{1}{(1-q^k)^s}, \\
M^\tIV_k(\gth)=M^{1,\tIV}_k(\gth):=&\, \frac{q^k}{(1-q^k)},\quad
M^{1,\tIV}_k(z_{s}):=\frac{q^{k(s-1)}}{(1-q^k)^s},\quad M^\tIV_k(z_{s}):=\frac{q^{sk}}{(1-q^k)^s}.
\end{align*}

The following theorem is parallel to \cite[Proposition 9]{CEM2013}
and includes \cite[Theorem 1]{Takeyama2013}.
\begin{thm}\label{thm:stuffleForAll}
Let $\gt=\tI$, \II, \III or $\tIV$. For any $\bfu_\gt,\bfv_\gt\in\fA_\gt^0$ we have
\begin{align}\label{equ:stuffleForAll}
\frakz_{q}[\bfu_\gt *_\gt\bfv_\gt]=&\, \frakz_{q}[\bfu_\gt]\frakz_{q}[\bfv_\gt].
\end{align}
\end{thm}
\begin{proof}
Since type $\tI$ case is just \cite[Theorem 1]{Takeyama2013},
we only need to consider the other three types.
The proof is basically the same as that of \cite[Theorem 1]{Takeyama2013}.
In fact, it suffices to observe that
\begin{alignat*}{3}
&\, M^\II_m(z'_{k})M^\II_m(z'_{l})=M^\II_m(z'_{k+l}),\quad & &
M^{1,\III}_m(z'_{k})M^\III_m(z_{l})=M^\III_m(z_{k+l}-z_{k+l-1}),\\
&\, M^\III_m(z_{k})M^\III_m(z_{l})=M^\III_m(z_{k+l}),\quad& &
M^{1,\III}_m(z'_{k})M^{1,\III}_m(z'_{l})=M^{1,\III}_m(z'_{k+l}-z'_{k+l-1}), \\
&\, M^\tIV_m(z_{k})M^\tIV_m(z_{l})=M^\tIV_m(z_{k+l}),\quad& &
M^{1,\tIV}_m(z_{k})M^\tIV_m(z_{l})=M^\tIV_m(z_{k+l}+z_{k+l-1}),  \\
&\, M^\tIV_m(\gth)M^\tIV_m(z_{k})=M^\tIV(z_{k+1}), \quad& &
M^{1,\tIV}_m(z_{k})M^{1,\tIV}_m(z_{l})=M^{1,\tIV}_m(z_{k+l}+z_{k+l-1}),\\
&\, M^{1,\tIV}_m(\gth)M^\tIV_m(z_{k})=M^\tIV(z_{k+1}),\quad& &
 M^\tIV_m(\gth)M^\tIV_m(\gth)= M^{1,\tIV}_m(\gth)M^\tIV_m(\gth)=M^\tIV_m(z_2), \\
&\, M^{1,\tIV}_m(\gth)M^{1,\tIV}_m(z_{k})=M^{1,\tIV}_m(z_{k+l}),\quad& &
M^{1,\tIV}_m(\gth)M^{1,\tIV}_m(\gth)=M^{1,\tIV}_m(z_2-\gth),
\end{alignat*}
for all $k, l\ge 0,\ m\ge 1$. Of course, we need to assume $k,l\ge 2$ for
$M^{1,\tIV}_m(z_{k})$ and $M^{1,\tIV}_m(z_{l})$.
\end{proof}

\section{Iterated Jackson's {\lowercase{$q$}}-integrals}\label{sec:Jint}
Set
\begin{equation*}
	x_0:=x_0(t)=\frac{1}{t}, \quad x_1:=x_1(t)=\frac{1}{1-t}, \quad \bfy:=\bfy(t)=\frac{t}{1-t}.
\end{equation*}
Recall that for $a=x_0(t) dt$ and $b=x_1(t) dt$, we can express MZVs by Chen's iterated integrals:
\begin{equation*}
\zeta(s_1,\dots,s_d)=\int_0^1 a^{s_1-1}b\cdots a^{s_d-1}b.
\end{equation*}
Replacing the Riemann integrals by the Jackson $q$-integrals \eqref{equ:Jackson} one gets

\begin{thm}\label{thm:qIteratedZetaIIICEM}\emph{(\cite[(29)]{CEM2013})}
For $\bfs=(s_1,\dots,s_d)\in\N^d$ set $w=|\bfs|$ and
\begin{equation*}
\tilde\zeta_q^\III[\bfs,t]
:=\bfJ\Big[ c_1 \bfJ\big[ c_2 \cdots \bfJ[c_w] \cdots \big] \Big](t),	
\end{equation*}		
where $c_i=x_1$ if $i \in \{u_1,u_2,\ldots,u_d\}$, $u_j:=s_1+s_2+\cdots+s_j$,
and $c_i=x_0$ otherwise. Or, equivalently,
set $\bfw=\pi^{s_1}y\pi^{s_2} y \dots \pi^{s_d} y$ and
\begin{equation*}
     \tilde\frakz_q^\III[\bfw;t]:=\bfP^{s_1}\big[\bfy\cdots\bfP^{s_d}[\bfy]\cdots \big](t).
\end{equation*}
Then 	
\begin{equation*}
\frakz_q^\III[\bfs]=\tilde\frakz_q^\III[\bfw;q]
\end{equation*}
\end{thm}

However, the representation of $\zeta_q^\III[\bfs]$ using $\tilde\frakz_q^\III$
in Theorem~\ref{thm:qIteratedZetaIIICEM}
is not ideal in the sense that one has to
evaluate $t$ at $q$. We would like to use Corollary~\ref{cor:permuteRy} so we need to
set $t=1$. This leads to the idea of replacing the first factor $\bfP^{s_1}$ by $\bfP^{s_1-1}\bfR$ and,
more generally, the following two generalizations.

\begin{thm}\label{thm:qIteratedZetaGen1}
Let $\bfs=(s_1,\dots,s_d)\in\N^d$ and $\bfa=(a_1,\dots,a_d)\in(\Z_{\ge 0})^d$. Put
$w=|\bfs|$ and $\bfw=\bfw^\bfa(\bfs)=\rho^{a_1}\pi^{s_1-a_1} y\dots \rho^{a_d} \pi^{s_d-a_d} y$.
Define
\begin{equation*}
\frakz_q[\bfw^{\bfa}(\bfs);t]
=\bfR^{a_1}\big[\bfP^{s_1-a_1}[\bfy\bfR^{a_2}[\bfP^{s_2-a_2}[\bfy\cdots
    \bfR^{a_d}[\bfP^{s_d-a_d}[\bfy]]\cdots]]]\big](t).
\end{equation*}
If $a_1+\dots+a_j>0$ for all $j=1,\dots,d$, then we have
\begin{equation} \label{equ:JexpressZetaGen1}
 \zeta_q^{\bfa}[\bfs]
=(1-q)^w\frakz_q[\bfw^{\bfa}(\bfs);\, 1], \qquad
\frakz_q^{\bfa}[\bfs]=\frakz_q[\bfw^{\bfa}(\bfs)]:=\frakz_q[\bfw^{\bfa}(\bfs);\, 1].
\end{equation}
\end{thm}

\begin{proof} First we observe three important facts: for any $k\ge 1$ we have
\begin{equation*}
     \bfP(t^k)=\sum_{j\ge 0} q^{kj} t^k=\frac{t^k}{1-q^k},\quad
     \bfD(t^k)=t^k(1-q^k),
     \quad\text{and}\quad
     \bfR(t^k)=\sum_{j\ge 1} q^{kj} t^k=\frac{q^k t^k}{1-q^k}
\end{equation*}
by the definition of the two summation operators and the difference operator.
Repeatedly applying this we get
\begin{alignat}{3}\label{equ:PpowerOntPower}
     \bfP^m(t^k)=&\, \sum_{j\ge 0} q^{kj} t^k=\frac{t^k}{(1-q^k)^m}, \qquad &&\forall m\in \Z, \\
     \bfR^m(t^k)=&\, \sum_{j\ge 1} q^{kj} t^k=\frac{q^{mk} t^k}{(1-q^k)^m}  \qquad &&\forall m\in \Z_{\ge 0}.
     \label{equ:RpowerOntPower}
\end{alignat}
Thus
\begin{equation*}
     \bfP\big(\bfy(t)\cdot t^k \big)
=\sum_{j\ge 0} \frac{q^{j(k+1)} t^{k+1}}{1-q^{j}t}
=\sum_{j\ge 0}\sum_{\ell\ge 0}  q^{j(k+\ell+1)} t^{k+\ell+1}
=\sum_{\ell>k}  \frac{t^{\ell}}{1-q^\ell}.
\end{equation*}
Similarly, we have
\begin{equation*}
     \bfD\big(\bfy(t)\cdot t^k \big)
=  \frac{t^{k+1} } {1-t}-\frac{q^{k+1} t^{k+1}} {1-qt}
=\sum_{\ell\ge 0} (1- q^{k+\ell+1}) t^{k+\ell+1}
=\sum_{\ell>k} (1-q^\ell) t^\ell,
\end{equation*}
and
\begin{equation*}
     \bfR\big(\bfy(t)\cdot t^k \big)
=\sum_{j\ge 1} \frac{q^{j(k+1)} t^{k+1}}{1-q^{j}t}
=\sum_{j\ge 1}\sum_{\ell\ge 0}  q^{j(k+\ell+1)} t^{k+\ell+1}
=\sum_{\ell>k}  \frac{q^\ell t^{\ell}}{1-q^\ell}.
\end{equation*}
It follows from \eqref{equ:PpowerOntPower} and \eqref{equ:RpowerOntPower} that
\begin{alignat}{3}\label{equ:PDpowerOny(t)tPower}
\bfP^m\big(\bfy(t)\cdot t^k \big)=&\, \sum_{\ell>k}  \frac{t^{\ell}}{(1-q^\ell)^m}
     \qquad &&\forall m\in \Z, \\
\bfR^m\big(\bfy(t)\cdot t^k \big)=&\,\sum_{\ell>k}  \frac{q^{m\ell} t^{\ell}}{(1-q^\ell)^m}
     \qquad &&\forall m\in \Z_{\ge 0}.
     \label{equ:RpowerOny(t)tPower}
\end{alignat}

We now prove by induction on the the depth $d$ that for all $\bfs=(s_1,\dots,s_d)\in\N^d$,
\begin{equation}\label{equ:InductionThmZetaGen1}
\frakz_q[\bfw^{\bfa}(\bfs);t]=\sum_{k_1>\dots>k_d> 0} \frac{t^{k_1}  q^{k_1a_1}\dots q^{k_da_d}}
            {(1-q^{k_1})^{s_1}\cdots (1-q^{k_d})^{s_d}}.	
\end{equation}
When $d=1$, i.e., $\bfs=s$, then by \eqref{equ:PDpowerOny(t)tPower} followed by \eqref{equ:RpowerOntPower}
\begin{equation*}
\frakz_q[\bfw^{a}(s);t]= \bfR^a\bfP^{s-a}[\bfy](t)=\sum_{k>0}\frac{\bfR^a\big(t^k\big)}{(1-q^k)^{s-a}}
 =\sum_{k>0}\frac{q^{ak} t^k}{(1-q^k)^s}
 =\frakz_q^{(a)}[s;t].
\end{equation*}
This proof works even when $s=a$ because of \eqref{equ:RpowerOny(t)tPower} (take $k=0$ and $m=a$ there).

In general, assuming $d\ge 2$ and \eqref{equ:InductionThmZetaGen1} is true for smaller depths.
Then by the inductive assumption
\begin{align*}
\frakz_q[\bfw^{\bfa}(\bfs);t]=&\, \bfR^{a_1}\bfP^{s_1-a_1}\big[\bfy \bfR^{a_2}\bfP^{s_2-a_2}[\bfy \cdots \bfR^{a_d}\bfP^{s_d-a_d}[\bfy ]\cdots]\big](t) \\
=&\,  \sum_{k_2>\dots>k_d> 0}
\frac{ \bfR^{a_1}\bfP^{s_1-a_1}\big(\bfy(t)\cdot t^{k_2} \big) q^{k_2 a_2} \dots q^{k_da_d}}
    {(1-q^{k_2})^{s_2}\cdots (1-q^{k_d})^{s_d}} \\
=&\, \sum_{k_1>\dots>k_d> 0}
\frac{ \bfR^{a_1}\big(t^{k_1}\big) q^{k_2 a_2} \dots q^{k_da_d}}
    {(1-q^{k_1})^{s_1-a_1}(1-q^{k_2})^{s_2}\cdots (1-q^{k_d})^{s_d}}
        \quad (\text{by \eqref{equ:PDpowerOny(t)tPower})}\\
=&\, \sum_{k_1>\dots>k_d> 0} \frac{t^{k_1} q^{k_1 a_1} \dots q^{k_da_d}}
 {(1-q^{k_1})^{s_1}\cdots (1-q^{k_d})^{s_d}} 	 	
\end{align*} 
by \eqref{equ:RpowerOntPower}. Again, if $s_1=a_1$ the proof is still valid.
This completes the proof of \eqref{equ:InductionThmZetaGen1}. Setting $t=1$
we arrive at \eqref{equ:JexpressZetaGen1}.
\end{proof}

By change of variables $a_j\to s_j-a_j$ for all $j=1,\dots,d$ we immediately obtain
the next result.
\begin{thm}\label{thm:qIteratedZetaGen2}
For $\bfs=(s_1,\dots,s_d)\in\N^d$ and $\bfa=(a_1,\dots,a_d)\in(\Z_{\ge 0})^d$, we set
$\bfs-\bfa=(s_1-a_1,\dots,s_d-a_d)$, $w=|\bfs|$ and
$\bfw^{\bfs-\bfa}(\bfs)=\rho^{s_1-a_1} \pi^{a_1} y  \dots \rho^{s_d-a_d} \pi^{a_d} y$.
Define
\begin{equation*}
\frakz_q[\bfw^{\bfs-\bfa}(\bfs);t]
=\bfR^{s_1-a_1}\big[\bfP^{a_1}[\bfy\bfR^{s_2-a_2}[\bfP^{a_2}[\bfy\cdots
    \bfR^{s_d-a_d}[\bfP^{a_d}[\bfy]]\cdots]]]\big](t).
\end{equation*}
If $s_1+\dots+s_j>a_1+\dots+a_j$ for all $j=1,\dots,d$, then we have
\begin{equation} \label{equ:JexpressZetaGen2}
 \zeta_q^{\bfs-\bfa}[\bfs]
=(1-q)^w\frakz_q[\bfw^{\bfs-\bfa}(\bfs);\, 1], \qquad
\frakz_q^{\bfs-\bfa}[\bfs]=\frakz_q[\bfw^{\bfs-\bfa}(\bfs)]:=\frakz_q[\bfw^{\bfs-\bfa}(\bfs);\, 1].
\end{equation}
\end{thm}

By specializing the proceeding two theorems to the four types of $q$-MZVs
in Table~\ref{Table:qMZV} we quickly find the following corollary.
For future reference, we will say $\bfw_\gt$ has
the \emph{typical type $\gt$ form} for each type $\gt$.

\begin{cor}\label{cor:qIteratedZetaAll}
For $\bfs=(s_1,\dots,s_d)\in\N^d$, we set
\begin{align*}
\bfw_\I=&\, \rho^{s_1-1} \pi y  \dots \rho^{s_d-1} \pi y
=z_{s_1}\dots z_{s_d}\in \fA_\gth^{(B)}\subset  \fA_\pi y  \qquad (s_1\ge2),\\
\bfw_\II=&\, \rho^{s_1} y\dots\rho^{s_d} y
=z'_{s_1}\dots z'_{s_d}\in \fA_\gth^{(B)}\subset\fA_\pi y  ,\\
\bfw_\III=&\, \pi^{s_1-1}\rho y\pi^{s_2} y \dots \pi^{s_d} y
=z'_{s_1}z_{s_2}\dots z_{s_d}\in \fA_\gth^{(A)}\subset \fA_\pi y , \\
\bfw_\IV=&\, \rho^{s_1-1}\pi y\rho^{s_2} y\dots\rho^{s_d} y
=z_{s_1}z'_{s_2}\dots z'_{s_d}\in \fA_\gth^{(B)}\subset\fA_\pi y  \qquad (s_1\ge2),
\end{align*}
and
\begin{align*}
\frakz_q[\bfw_\I;t]=&\, \bfR^{s_1-1}\big[\bfP[\bfy\bfR^{s_2-1}[\bfP[\bfy\cdots\bfR^{s_d-1}[\bfP[\bfy]]\cdots]]]\big](t), \\
\frakz_q[\bfw_\II;t]=&\, \bfR^{s_1}\big[\bfy\bfR^{s_2}[\bfy\cdots\bfR^{s_d}[\bfy]\cdots]\big](t), \\
\frakz_q[\bfw_\III;t]:=&\, \bfP^{s_1-1}\Big[\bfR\big[\bfy[\bfP^{s_2}[\bfy[\bfP^{s_3}[\bfy \cdots\bfP^{s_d}[\bfy]\cdots ]]]]\big]\Big](t), \\
\frakz_q[\bfw_\IV;t]=&\, \bfR^{s_1-1}\Big[\bfP\big[\bfy\bfR^{s_2}[\bfy\bfR^{s_3}[\bfy\cdots\bfR^{s_d}[\bfy]\cdots]]\big]\Big](t).
\end{align*}
Then for all the types $\gt=\I,$ \II, \III and \IV, we have	
\begin{equation*}
 \zeta_q^\gt[\bfs]
= (1-q)^w\frakz_q[\bfw_\gt;\, 1], \qquad
\frakz_q^\gt[\bfs]=\frakz_q[\bfw_\gt]:=\frakz_q[\bfw_\gt;\, 1].
\end{equation*}
Moreover, similar results hold for type $\tI$ and $\tIV$ $q$-MZVs.
We may replace any of the consecutive strings $\rho^{s_j-1}\pi$
by a single $\rho$ in $\bfw_\tI$ and $\bfw_\tIV$, and replace the
corresponding operator string $\bfP^{s_j-1}\bfR$ by a single $\bfR$.
\end{cor}

We now apply the above to Okounkov's $q$-MZVs. For any $n\in\N$ we let
$n^-$ and $n^+$ be the two nonnegative integers such that
\begin{equation*}
   \frac{n-1}2\le n^- \le   \frac{n}2 \le n^+\le  \frac{n+1}2.
\end{equation*}
Clearly we have $n^+ +n^-=n$ always, $n^+=n^-$ if $n$ is even, and $n^+=n^-+1$ if $n$ is odd.
We can now define a variation of Okounkov's $q$-MZVs. Let $\bfs\in (\Z_{\ge 2})^d$. Then
\begin{equation*}
\zeta_q^\O[\bfs]
:=\sum_{k_1>\dots>k_d>0} \prod_{j=1}^d \frac{q^{k_j^+s_j}+ q^{k_j^- s_j} }{[k_j]^{s_j}}
=(1-q)^{|\bfs|}\sum_{ k_1>\dots>k_d>0} \prod_{j=1}^d \frac{q^{k_j^+s_j}+ q^{k_j^- s_j} }{(1-q^{k_j})^{s_j}}.
\end{equation*}
Again, its modified form is:
\begin{equation*}
\frakz_q^\O[\bfs]:=\sum_{ k_1>\dots>k_d>0}
        \prod_{j=1}^d \frac{q^{k_j^+s_j}+ q^{k_j^- s_j} }{(1-q^{k_j})^{s_j}}.
\end{equation*}

\begin{rem}
The above variation is equal to Okounkov's original $q$-MZVs up to a suitable 2-power.
More precisely, the power is given by the number of even arguments in $\bfs$.
\end{rem}

\begin{cor}\label{cor:qIteratedZetaOk}
For $\bfs=(s_1,\dots,s_d)\in\N^d$, we set
\begin{align*}
\bfw_\O=&\, (\rho^{s_1^-}\pi^{s_1^+}+\rho^{s_1^+}\pi^{s_1^-}) y  \dots
(\rho^{s_d^-}\pi^{s_d^+}+\rho^{s_d^+}\pi^{s_d^-}) y
=z_{s_1}\dots z_{s_d}\in \fA_\gth^{(B)}\subset  \fA_\pi y
\end{align*}
and
\begin{align*}
\frakz_q[\bfw_\O;t]=&\, (\bfR^{s_1^-}\bfP^{s_1^+}+\bfR^{s_1^+}\bfP^{s_1^-})\big[\bfy
\cdots(\bfR^{s_d^-}\bfP^{s_d^+}+\bfR^{s_d^+}\bfP^{s_d^-})[\bfy]\cdots\big](t) .
\end{align*}
Then we have	
\begin{equation*}
 \zeta_q^\O[\bfs]
= (1-q)^w\frakz_q[\bfw_\O;\, 1], \qquad
\frakz_q^\O[\bfs]=\frakz_q[\bfw_\O]:=\frakz_q[\bfw_\O;\, 1].
\end{equation*}
\end{cor}

It is possible to obtain the shuffle relations among $\frakz_q^\O[\bfs]$-values using
Corollary~\ref{cor:qIteratedZetaOk}.
The stuffle relations among $\frakz_q^\O[\bfs]$ is mentioned implicitly in Okounkov's
original paper. For our modified version, they
can be derived from the following fact (cf.\ Proposition 2.2 (ii) of \cite{BachmannKu2014}).
Let $F^\O_n(t)=(t^{n^+}+t^{n^-})/(1-t)^n$ for all $n\ge 2.$
Then for all $r,s \in \Z_{\ge 2}$, we have
\begin{equation*}
F^\O_r(t) \cdot F^\O_s(t) =
\left\{
  \begin{array}{ll}
    2F^\O_{r+s}(t), & \hbox{if $r$ or $s$ is even;} \\
    2F^\O_{r+s}(t) + \frac12 F^\O_{r+s-2}(t), & \hbox{if $r$ and $s$ are odd.}
  \end{array}
\right.
\end{equation*}
For example,
\begin{align*}
\frakz_q^\O[2,3]\frakz_q^\O[2]=&\, 2\frakz_q^\O[2,2,3]+\frakz_q^\O[2,3,2]+2\frakz_q^\O[4,3]+2\frakz_q^\O[2,5],\\
\frakz_q^\O[2,3]\frakz_q^\O[3]=&\, 2\frakz_q^\O[2,3,3]+\frakz_q^\O[3,2,3]+2\frakz_q^\O[5,3]+2\frakz_q^\O[2,6]+\frac12\frakz_q^\O[2,4].
\end{align*}

\section{{\lowercase {$q$}}-shuffle relations} \label{sec:qShuffle}
In contrast to the MZV case, the $q$-shuffle product is much more difficult to
define than the $q$-stuffle product. In this section we will use
the Rota-Baxter algebra approach to define this for type $\tI$, \II, \III, and $\tIV$ $q$-MZVs.
Note that this has been done for type \III $q$-MZVs in \cite{CEM2013} which
we recall first.

The \emph{$q$-shuffle product} on $\fA_\pi$
is defined recursively as follows: for any words $\bfu,\bfv\in X_\pi^*$ we define
$\myone\sha\bfu=\bfu\sha\myone=\bfu$ and
\begin{align}
 (y\bfu)\sha\bfv=\bfu\sha(y\bfv)=&\,y(\bfu\sha\bfv),	\label{shuffle-y}\\
\pi\bfu\sha\pi\bfv=&\,\pi (\bfu\sha\pi\bfv)
        +\pi (\pi\bfu\sha\bfv)-\pi (\bfu\sha\bfv),	\label{shuffle-pipi}\\
\gd\bfu\sha\gd\bfv=&\,\bfu\sha\gd\bfv+\gd\bfu\sha\bfv -\gd(\bfu\sha\bfv),		 \label{shuffle-gdgd}\\
\gd\bfu\sha\pi\bfv=\pi\bfv\sha\gd\bfu=&\,\gd(\bfu\sha\pi\bfv)+\gd\bfu\sha\bfv -\bfu\sha\bfv		 \label{shuffle-gdpi}	
\end{align}
for any words $\bfu,\bfv\in X_\pi^*$. The first equation reflects the fact that when $\bfy(t)$ is
multiplied in front of either of the two factors in a product, it can be multiplied after taking the product.
The other  equations formalize \eqref{equ:PisRBalg},\eqref{equ:RisRBalg},
\eqref{equ:DDproduct}, \eqref{equ:DPproduct}, and \eqref{equ:DRproduct}, respectively.

\begin{cor} \label{cor:rhoId}
For any words $\bfu,\bfv\in X_\pi^*$, we have
{\allowdisplaybreaks
\begin{align}
\rho\bfu\sha\rho\bfv=&\,\rho (\bfu\sha\rho\bfv)
   +\rho (\rho\bfu\sha\bfv)+\rho (\bfu\sha\bfv), \label{shuffle-rhorho}\\
\rho\bfu\sha\pi\bfv=\pi\bfv\sha\rho\bfu
    =&\,\rho(\rho\bfu\sha\bfv)+\rho(\bfu\sha\rho\bfv)+\rho\bfu\sha\bfv+\rho(\bfu\sha\bfv),	 \label{shuffle-rhopi}\\
\gd\bfu\sha\rho\bfv=\rho\bfv\sha\gd\bfu=&\,\gd(\bfu\sha\pi\bfv)-\bfu\sha\bfv
=\gd(\bfu\sha\rho\bfv)+\gd(\bfu\sha\bfv) -\bfu\sha\bfv.			 \label{shuffle-gdrho}
\end{align}}
\end{cor}
\begin{proof}
These follows easily from \eqref{shuffle-y}--\eqref{shuffle-gdpi} and the relation $\rho=\pi-\myone.$
\end{proof}

\begin{cor}\label{cor:consistentWithTakeyama} 
For $j=1,2$ let $X_\gth^{(j)}$ and $X_\gth^{(j),*}$ be the embedding
of $X_\gth$ and $X_\gth^*$ into $X_\pi^*$, respectively, by Definition~\ref{defn:q-AlgebrafApi}.
For any $\ga,\gb\in X^{(j)}_\gth$ and $\bfu,\bfv \in X_\gth^{(j),*}$,
we have $1\sha \bfu=\bfu \sha 1=\bfu$ and
\begin{equation*}
\ga\bfu\sha \gb\bfv=\ga(\bfu\sha \gb\bfv)+\gb(\ga\bfu\sha\bfv)+[\ga,\gb]_j (\bfu\sha\bfv),
\end{equation*}
where $[\ga,\gb]_j$ is determined by $[a,b]_1=[b,a]_1=-b$, $[a,b]_2=[b,a]_2=0$ and
\begin{equation}\label{eq:rho-prod}
[a,a]_j=(-1)^j a, \quad
[b,b]_j=-b\gam, \quad
[\ga,\gam]_j=[\gam,\ga]_j=-\ga\gam.
\end{equation}
\end{cor}
\begin{proof}
All of these identities follow from straight-forward computation using
\eqref{shuffle-y}--\eqref{shuffle-gdrho}. For example,
\begin{equation} \label{equ:bbShuffle}
\begin{split}
b\bfu\sha b\bfv=\pi y\bfu\sha \pi y\bfv
= &\, \pi (y\bfu\sha \pi y\bfv)+\pi(\pi y\bfu\sha y\bfv)-\pi (y\bfu\sha y\bfv)\\
= &\, \pi y(\bfu\sha \pi y\bfv)+\pi y(\pi y\bfu\sha\bfv)-\pi y y(\bfu\sha\bfv)\\
= &\, b(\bfu\sha b\bfv)+b(b\bfu\sha\bfv)-b\gam(\bfu\sha\bfv).
\end{split}
\end{equation}
Similarly,
\begin{equation} \label{equ:gthgthShuffle}
\begin{split}
\gth\bfu\sha \gth\bfv=\rho y\bfu\sha \rho y\bfv
= &\, \rho (y\bfu\sha \rho y\bfv)+\rho(\rho y\bfu\sha y\bfv)+\rho (y\bfu\sha y\bfv)\\
= &\, \rho y(\bfu\sha \rho y\bfv)+\rho y(\rho y\bfu\sha\bfv)+\rho y y(\bfu\sha\bfv)\\
= &\, \gth(\bfu\sha \gth\bfv)+\gth(\gth\bfu\sha\bfv)+\gth\gam(\bfu\sha\bfv).
\end{split}
\end{equation}

The rest of the proof is left to the interested reader.
\end{proof}

\begin{prop}\label{prop:fAShaCommAss}
The algebra $(\fA_\pi,\sha)$ is commutative and associative.
\end{prop}
\begin{proof}
See \cite[Theorem~7]{CEM2013}.
\end{proof}

The following corollary generalizes \cite[Proposition 1]{Takeyama2013}.
\begin{cor}\label{cor:fAgthShaCommAss}
For $j=1$ or $2$ the algebras $(\fA^{(j)}_\gth,\sha)$ are commutative and associative.
\end{cor}
\begin{proof}
This follows immediately from Proposition~\ref{prop:fAShaCommAss}
since $(\fA^{(j)}_\gth,\sha)$ are sub-algebras of $(\fA_\pi,\sha)$
if $\sha$ for $\fA^{(j)}_\gth$ is defined as in
Corollary~\ref{cor:consistentWithTakeyama}.
\end{proof}

Our next theorem shows that we may use the shuffle algebra structure defined above to
describe the shuffle relations among different types of $q$-MZVs. Before doing so, we need
to show that for each type the shuffle product really makes sense.

\begin{prop}\label{equ:shuClosedForAll}
Embed $\fA_\tI^0,\fA_\II^0,\fA_\tIV^0\subset\fA_\gth^{(B)}$ and $\fA_\III^0 \subset \fA_\gth^{(A)}$.
Then for each type $\gt$, if the two words $\bfu,\bfv \in \fA_\gt^0$ have
the typical type $\gt$ form listed in Corollary~\ref{cor:qIteratedZetaAll} then
there is an algorithm to express $\bfu\sha\bfv$ using only those words in the same form.
\end{prop}
\begin{proof}
The case for type $\tI$ is proved by \cite[Proposition~2.4]{Takeyama2013}.

Type \II is in fact the easiest since we can restrict ourselves to use only
\eqref{shuffle-y} and \eqref{shuffle-rhorho} to compute the shuffle and
therefore $\pi$ never comes into the picture. Clearly all such words must
start with $\rho$ and end with $y$.

For type \III let's assume $\bfu=\pi^{s_1-1}\rho y\pi^{s_2} y \dots \pi^{s_d} y$
and $\bfv=\pi^{t_1-1}\rho y\pi^{t_2} y \dots \pi^{t_d} y$.
If we use the definition \eqref{shuffle-pipi} repeatedly then in each word appearing in
$\bfu\sha\bfv$ the first $\rho$ always appears before all the $y$'s. Such a word
can be written in the form $\pi^s \rho^r y \cdots$ for some $s\in\Z$ and $r\ge 1$
(notice that if $\rho$ and $\pi$ are commutative). Now we can
rewrite this as $\pi^s(\pi-\myone)^{r-1} \rho  y \cdots$ and replace all the $\rho$'s
after the first $y$ by $\pi-\myone$. This produces a word of typical type \III form.

Type $\tIV$ is similar to type \III except that we need to take $\gth$ into account.
Notice that by definition if $\bfw \in \fA_\IV^0$ then it can be written as
$\gth \bfw'$, or $z_k\bfw'$ ($k\ge 2$, $\bfw'\in Y_\tIV^*$) or a finite
linear combination of these. So we have three cases to check.
First, we prove that for all $k,l\ge 2$ and $\bfu,\bfv\in Y_\tIV^*$
\begin{equation}\label{equ:typeIIIshuClosedCase1}
      z_k\bfu \sha z_l\bfv \in \fA_\tIV^0.
\end{equation}
Indeed, putting $k=r+1$ and $l=s+1$ we have
\begin{equation*}
\rho^r \pi y \bfu \sha \rho^s \pi y \bfv
=\rho(\rho^{r-1} \pi y \bfu \sha \rho^s \pi y \bfv)
+\rho(\rho^r\pi y \bfu \sha \rho^{s-1} \pi y \bfv)
+\rho(\rho^{r-1} \pi y \sha \rho^{s-1} \pi y \bfv).
\end{equation*}
Now inside each of the three parentheses we replace every $\pi$ by $\rho+\myone$ and
use  only \eqref{shuffle-y} and \eqref{shuffle-rhorho}
to expand (recall that $\gth=\rho y$). We see that every term in the expansion has the form
$\rho^n y \bfw$ for some $n\ge 1$ and $\bfw\in Y_\IV^*$. If $n=1$ then we have
$\rho^n y\bfw =\gth\bfw \in \fA_\tIV^0.$
If $n\ge 2$ we can write it as
\begin{equation}\label{equ:phikrhoykyCompare}
\rho^{n-1} (\pi-\myone) y \bfw
=\sum_{j=1}^{n-1} (-1)^{j-1} \rho^{n-j} \pi y \bfw +(-1)^{n-1} \gth\bfw\in \fA_\IV^0
\end{equation}
with each word of typical type $\tIV$ form.

Now we assume $k=r+1\ge 2$ and $\bfu,\bfv\in Y_\tIV^*$. Then
\begin{equation*}
  z_k\bfu \sha \gth \bfv
=\rho^r \pi y \bfu \sha \rho y \bfv
=\rho(\rho^{r-1} \pi y \bfu \sha \rho y \bfv)
+\rho y(\rho^r\pi y \bfu \sha \bfv)
+\rho y(\rho^{r-1} \pi y \sha \bfv)\in \fA_\IV^0
\end{equation*}
since $\rho y=\gth$ and the first term can be dealt with as in the proof of
\eqref{equ:typeIIIshuClosedCase1}.

Finally,
\begin{equation*}
  \gth \bfu \sha \gth \bfv \in \fA_\tIV^0
\end{equation*}
follows from \eqref{equ:gthgthShuffle} immediately. This completes the proof of the proposition.
\end{proof}

The following theorem generalizes \cite[Theorem 2]{Takeyama2013} but it does not contain
\cite[Theorem 7]{CEM2013} since our word representation of type \III $q$-MZVs is different
from that given in \cite{CEM2013}.
\begin{thm}\label{thm:qZetaShuffleAll}
Embed $\fA_\tI^0,\fA_\II^0,\fA_\tIV^0\subset\fA_\gth^{(B)}$ and $\fA_\III^0 \subset \fA_\gth^{(A)}$.
Then for each type $\gt$ and for any $\bfu_\gt,\bfv_\gt \in \fA_\gt^0$, we have
\begin{equation}\label{equ:qZetaShuffleAll}
\frakz_q[\bfu_\gt]\frakz_q[\bfv_\gt]=\frakz_q[\bfu_\gt \sha\bfv_\gt].
\end{equation}
\end{thm}

\begin{proof}
For each type $\gt$ we observe that $\frakz_q[\bfw_\gt;t]$
satisfies \eqref{equ:qZetaShuffleAll} because of the identities
in Proposition~\ref{prop:checkRBalg}.
Then the theorem follows from the fact that
$\frakz_q[\bfw_\gt]=\frakz_q[\bfw_\gt;\, 1]$ for any word $\bfw_\gt\in\fA_\gt^0$
by Corollary~\ref{cor:qIteratedZetaAll}.
\end{proof}

\section{Duality Relations}
The DBSFs do not contain all linear relations among the various types of $q$-MZVs.
In \cite{Takeyama2013}, Takeyama discovered the following relations which
provides some of the missing relations for type $\tI$ $q$-MZVs,
at least in the small weight cases. He called them Resummation Identities.
We would rather call them ``duality" relations because of their similarity to
the duality relations for the ordinary MZVs.

\begin{thm}
\label{thm:resummation-dualityI}  \emph{(\cite[Theorem 4]{Takeyama2013})}
For a positive integer $k$, set
\begin{equation*}
\varphi_{k}:=(-1)^k\left(\sum_{j=2}^{k}(-1)^j z_j-\gth\right),
\end{equation*}
where $\varphi_1=\gth=\rho y\in \fA_\gth^{(B)}$.
Let $\ell\in\N$ and $\ga_j,\gb_j\in\Z_{\ge 0}$ for all $j=1,\dots,\ell$.
Then we have
\begin{equation}
\zeta_q^\tI[\varphi_{\ga_{1}+1}\gamma^{\gb_{1}}\cdots\varphi_{\ga_{\ell}+1}\gamma^{\gb_{\ell}}]
=\zeta_q^\tI[\varphi_{\gb_{\ell}+1}\gamma^{\ga_{\ell}}\cdots\varphi_{\gb_{1}+1}\gamma^{\ga_{1}}].
\label{equ:resummation-duality}
\end{equation}
\end{thm}
We can use the Rota-Baxter algebra approach to give a new proof of this result.
\begin{proof}
Notice that $\gamma=y$, $z_j=\rho^{j-1} \pi y$ and $\gth=\rho y$ with the embedding $\fA_\tI^0\subset\fA_\gth^{(B)}$.
Since $\pi=\rho+\myone$, for all $k \ge 1$, we have (cf.~\eqref{equ:phikrhoykyCompare})
\begin{multline}\label{equ:phikrhoyky}
\varphi_{k}=(-1)^k\left(\sum_{j=2}^{k}(-1)^j \rho^{j-1} (\rho+\myone) y-\rho y\right)\\
= (-1)^k\left(\sum_{j=2}^{k}(-1)^j \rho^j y+\sum_{j=2}^{k}(-1)^j \rho^{j-1} y-\rho y\right)
=\rho^k y.
\end{multline}
Thus the theorem follows from Corollary~\ref{cor:permuteRy}
and Corollary~\ref{cor:qIteratedZetaAll} easily.
\end{proof}

\begin{rem}
Although not mentioned explicitly in \cite{Takeyama2013},
there is a subtle point in applying Theorem~\ref{thm:resummation-dualityI}.
Notice that in the expression of $\varphi_k$ the letter $\gth$ appears. However,
$q$-MZVs of the form such as
$\zeta_q^\I[\gth\gam z_2 \gam]=\zeta_q^\I[\rho y^2 \rho^2 \pi y]$
is not really defined. In fact, it should be denoted by
$\zeta_q^\tI[\gth\gam z_2 \gam]=\zeta_q^{(1,0,1,0)}[1,0,2,0]$
(and such values always converge by Proposition~\ref{prop:fAuxFunction}
because of the leading 1 in the auxiliary variable $\bft$).
But, suitable $\Q$-linear combinations of  \eqref{equ:resummation-duality} may lead
to identities in which only $z_k$'s appear.
Then all terms can be written as honest $\zeta_q^\I$-values.
This explains the use of two addmissible
structures $\widehat{\frakH}^0$ and $\frakH^0$ in \cite{Takeyama2013}.
For an illuminating example, see the proof of Proposition~7 of op.\ cit.
This remark also applies to Theorem~\ref{thm:resummation-dualityIV}
for the duality of type $\tIV$ $q$-MZVs.
\end{rem}

Similar relations for type \II $q$-MZVs have the most aesthetic appeal and
is the primary reason why we prefer to call it by ``duality''.
\begin{thm}
\label{thm:resummation-dualityIII}
Let $\ell\in\N$ and $\ga_j,\gb_j\in\N$ for all $j=1,\dots,\ell$.
Then we have
\begin{equation*}
\zeta_q^\II[\rho^{\ga_{1}}y^{\gb_{1}}\cdots\rho^{\ga_{\ell}}y^{\gb_{\ell}}]
=\zeta_q^\II[\rho^{\gb_{\ell}}y^{\ga_{\ell}}\cdots\rho^{\gb_{1}}y^{\ga_{1}}].
\end{equation*}
\end{thm}
\begin{proof}
This follows from Corollary~\ref{cor:permuteRy}
and Corollary~\ref{cor:qIteratedZetaAll} immediately.
\end{proof}

Of course we may apply the same idea to type \III and $\tIV$ $q$-MZVs.

\begin{thm}
\label{thm:resummation-dualityII}
Let $\ell\in\N$ and $\ga_j,\gb_j\in\N$ for all $j=1,\dots,\ell$.
Then we have
\begin{multline*}
\zeta_q^\III[(\pi-\myone)^{\ga_{1}-1}\rho y^{\gb_{1}}(\pi-\myone)^{\ga_{2}}y^{\gb_{2}}\cdots(\pi-\myone)^{\ga_{\ell}}y^{\gb_{\ell}}]\\
=\zeta_q^\III[(\pi-\myone)^{\gb_{\ell}-1}\rho y^{\ga_{\ell}}(\pi-\myone)^{\gb_{\ell-1}}y^{\ga_{\ell-1}}\cdots(\pi-\myone)^{\gb_{1}}y^{\ga_{1}}].
\end{multline*}
\end{thm}
\begin{proof}
Since $\rho=\pi-\myone$ this follows from Corollary~\ref{cor:permuteRy}
and Corollary~\ref{cor:qIteratedZetaAll} easily.
\end{proof}

\begin{thm}
\label{thm:resummation-dualityIV}
Let $\ell\in\N$ and $\ga_j,\gb_j\in\N$ for all $j=1,\dots,\ell$.
Then we have
\begin{equation*}
\zeta_q^\tIV[\varphi_{\ga_1} y^{\gb_{1}-1}\rho^{\ga_{2}}y^{\gb_{2}}\cdots\rho^{\ga_{\ell}}y^{\gb_{\ell}}]
=\zeta_q^\tIV[\varphi_{\gb_{\ell}} y^{\ga_{\ell}-1}\rho^{\gb_{\ell-1}}y^{\ga_{\ell-1}}\cdots\rho^{\gb_{1}}y^{\ga_{1}}].
\end{equation*}
Here $\varphi_1=\theta=\rho y\in\fA_\gth^{(B)}$.
\end{thm}
\begin{proof}
This follows from \eqref{equ:phikrhoyky}, Corollary~\ref{cor:permuteRy}
and Corollary~\ref{cor:qIteratedZetaAll}.
\end{proof}

\section{The general type G {\lowercase {$q$}}-MZVs} \label{sec:TypeG}
All of the $q$-MZVs of type $\tI$, \II, \III and $\tIV$ considered
in the above are some special forms of
the $q$-MZVs $\zeta_q^{(t_1,\dots, t_d)}[s_1,\dots, s_d]$ where
$1\le t_1\le s_1$, $0\le t_j\le s_j$ for all $j\ge 2$, all of which are convergent
by Proposition~\ref{prop:fAuxFunction}. We call these \emph{type G} $q$-MZVs.
Similar to the first four types, we may use words to encode these
values according to Theorem~\ref{thm:qIteratedZetaGen1} by setting $a_j=t_j$ there. Namely,
we can define
\begin{equation*}
\frakz_q[\rho^{a_1}\pi^{b_1} y \cdots \rho^{a_d}\pi^{b_d} y;t]
=\bfR^{a_1}\big[\bfP^{b_1}[\bfy\bfR^{a_2}[\bfP^{b_2}[\bfy\cdots
    \bfR^{a_d}[\bfP^{b_d}[\bfy]]\cdots]]]\big](t).
\end{equation*}
Then we have
\begin{equation*}
\frakz_q[\bfw^\bft(\bfs)]
:=\frakz_q[\bfw^\bft(\bfs);\, 1]
=\frakz_q^{(t_1,\dots, t_d)}[s_1,\dots, s_d],
\end{equation*}
where $\bfw^\bft(\bfs)=\rho^{t_1}\pi^{s_1-t_1} y \cdots \rho^{t_d}\pi^{s_d-t_d} y\in X_\pi^*$.
The shuffle product structure are reflected by $(X_\pi^*,\sha)$ where
the $\sha$ is defined by \eqref{shuffle-y}, \eqref{shuffle-pipi}, \eqref{shuffle-rhorho}
and \eqref{shuffle-rhopi}.

We observe that there is often more than one way to express a type G $q$-MZV
using words because of the relation $\pi=\rho+\myone.$ For example, using
the relations
\begin{equation*}
\pi^2 \rho y=\pi\rho^2 y+\pi\rho y=\rho^3 y+2\rho^2 y+\rho y
\end{equation*}
we get immediately the relations
\begin{equation*}
\zeta_q^{(1)}[3]=\zeta_q^{(2)}[3]+\zeta_q^{(1)}[2]
=\zeta_q^{(3)}[3]+2\zeta_q^{(2)}[2]+\zeta_q^{(1)}[1].
\end{equation*}
We call all such relations $\bfP$-$\bfR$ relations.

\begin{prop}
For all $\bfu, \bfv\in\fA^0_\GG$, we have $\bfu\sha \bfv\in\fA^0_\GG$.
\end{prop}
\begin{proof}
Notice that admissible words in $\fA^0_\GG$ must end with $y$
and have at least one $\rho$ before the first $y$.
Moreover, the converse is also true. This is rather
straight-forward if we use the $\bfP$-$\bfR$ relations repeatedly
to get rid of all the $\pi$'s.

Now, by using the definition of $\sha$ it is not hard to see that $\bfu\sha \bfv$
ends with $y$ and and has at least one $\rho$ before the first $y$
if both $\bfu$ and $\bfv$ are admissible. So $\bfu\sha \bfv\in\fA^0_\GG$
and the proposition is proved.
\end{proof}

To define the stuffle product we let
$$Y_\GG=\{z_{t,s} \mid t,s\in \Z_{\ge 0}, t\le s \},$$
and let $\fA_\GG$ be the the noncommutative polynomial $\Q$-algebra of words of $Y_\GG^*$ built
on the alphabet $Y_\GG$. Define the type $\GG$-admissible words as those in
\begin{equation*}
 \fA^0_\GG= \bigcup_{1\le t\le s} z_{t,s}\fA_\GG.
\end{equation*}
We can regard $\fA_\GG$ as a subalgebra of $X_\pi^*$ by setting $z_{t,s}=\rho^t\pi^{s-t} y$.
Then stuffle product $*_\GG$ on $\fA^0_\GG$ can be defined
inductively as follows. For any words $\bfu,\bfv\in\fA^0_\GG$ and
letters $z_{t,s},z_{t',s'}\in Y_\GG$ with $1\le t\le s$ and $1\le t'\le s'$
we set $\myone*_\GG \bfu=\bfu=\bfu*_\GG\myone$ and
\begin{equation*}
 (z_{t,s} \bfu)*_\GG(z_{t',s'}\bfv)=z_{t,s}(\bfu*_\GG z_{t',s'}\bfv)
    +z_{t',s'}(z_{t,s}\bfu*_\GG\bfv)+z_{t+t',s+s'} (\bfu*_\GG \bfv).
\end{equation*}
It is easy to show that $(\fA^0_\GG,*_\GG)$ is a commutative and associative algebra.

We leave the proof of the following theorems to the interested readers.
The first result clearly provides the DBSFs of type G $q$-MZVs.
\begin{thm}\label{thm:dbsfG}
For any $\bfu,\bfv\in\fA_\GG^0\subset X_\pi^*$ we have
\begin{equation}\label{equ:dbsfG}
\frakz_{q}[\bfu*_\GG \bfv]=\frakz_{q}[\bfu \sha \bfv]=\frakz_{q}[\bfu]\frakz_{q}[\bfv].
\end{equation}
\end{thm}

The duality relations are given in the cleanest form by Theorem~\ref{thm:resummation-dualityII}
which can be translated into the following.
\begin{thm}
\label{thm:resummation-dualityG}
Let $\ell\in\N$ and $\ga_j,\gb_j\in\N$ for all $j=1,\dots,\ell$.
Set
\begin{align*}
\bfs=&\, (\ga_1,0^{\gb_{1}-1},\ga_2,0^{\gb_{2}-1},\dots,\ga_\ell,0^{\gb_{\ell}-1}),\\
\bfs^\vee=&\, (\gb_\ell,0^{\ga_{\ell}-1},\gb_{\ell-1},0^{\ga_{\ell-1}-1},\dots,\gb_1,0^{\ga_{1}-1}).
\end{align*}
Then we have
\begin{equation*}
\zeta_q^{\bfs}[\bfs]=\zeta_q^{\bfs^\vee}[\bfs^\vee].
\end{equation*}
\end{thm}

\section{Numerical data}\label{sec:data}
In this last section, we compute the $\Q$-linear relations among various types of
$q$-MZVs of small weight by using the DBSFs and the duality relations. Most of the computation is
carried out with the computer algebra system MAPLE, version 16. My laptop has
Intel Core i7 with CPU speed at 2.4GHz and 16GB RAM.

For each type $\gt$ we will define the set of type $\gt$-admissible words $\sfW^\gt_{\le w}$
of weight and depth both bounded by $w$. This is necessary since we allow 0 in some
types of $q$-MZVs. We have to control the number of $0$'s occurring as
arguments in $q$-MZVs since otherwise the dimensions to be considered will become infinite.
Another reason that the depth has to be bounded is because
the duality essentially swaps the depth and the weight.

We denote by $\sfZ^\gt_{\le w}$ the $\Q$-space generated by $q$-MZVs of type $\gt$
corresponding to the type $\gt$-admissible words $\sfW^\gt_{\le w}$,
$\DS^\gt_{\le w}$ the space generated by the DBSFs, and
$\DU^\gt_{\le w}$ the space generated by the duality relations. Hence
$\DU^\gt_{\le w}\setminus \DS^\gt_{\le w}$ gives the duality relations that are
not contained in $\DS^\gt_{\le w}$.

\textbf{Type I}. We have seen that it is necessary to consider
$q$-MZVs of the form $\frakz_q^{\bft}(\bfs)$ with $(t_j,s_j)=(s_j-1,s_j)$ or
$(t_j,s_j)=(1,1)$. The latter case corresponds to the words containing the letter $\theta$.
We have called all of these values \emph{type $\tI$ $q$-MZVs}.

\begin{prop}
Let $F_{-1}=0$, $F_0=1$, $F_1=1$, ... be the Fibonacci sequence. Then
for all $w\ge 1$ we have
\begin{equation*}
\sharp\sfW^\I_{\le w}=2^{w-1}-1 \qquad\text{and}
\qquad \sharp\sfW^\tI_{\le w}=F_{2w}-1.
\end{equation*}
\end{prop}
\begin{proof}
The first equation follows from the same argument as that for MZVs. It is given
by the number of integer solutions to the inequality
\begin{equation*}
    s_1+\dots+s_d\le w,  \quad d\ge 1, s_1\ge 2, s_2,\dots,s_d\ge 1.
\end{equation*}
Or, more directly and perhaps much easier, we can count the corresponding
admissible words. Clearly,
there are $2^{w-1}$ ways to form a word consisting of $w-1$ letters where
the letters can be either $\rho$ or $\pi$. Let $S_w$ be the set of such words.
We now show that there is a one-to-one correspondence between $S_w$ and
the set $A_w$ of admissible word of weight $w$.
First, from each word $\bfw\in S_w$ we can obtain a word in $A_w$
by inserting a letter $y$ after each $\pi$ in $\bfw$ and attach $\pi y$ at the end.
On the other hand, for each word in $A_w$ we may chop off
the ending $\pi y$ and removing all the $y$'s to get a word in $\bfw$.
This establishes the one-to-one correspondence.

We now prove the second equation.
Let $a_n$ (resp.~$b_n$) be the number of type $\tI$ $q$-MZVs of weight $n$ beginning
with $(t_j,s_j)=(1,1)$ (resp.~$(t_j,s_j)=(s_j-1,s_j)$). Let's call the two different
beginnings 1-initial and 2-initial, respectively. Then $a_1=1$ and $b_1=0$.
Now to produce weight $n+1$ 1-initials one can attach  $(t,s)=(1,1)$ to
the beginning of any weight $n$ type $\tI$ $q$-MZVs. Moreover, one can
change the beginning of any weight $n$ 1-initial to $(t,s)=(0,1)$ and then
attach  $(t,s)=(1,1)$. Thus $a_{n+1}=2a_n+b_n$. To obtain 2-initials of weight $n+1$
one either changes a 1-initial of weight $n$ to begin with $(t,s)=(1,2)$ or changes a
2-initial value of weight $n$ to begin with $(s,s+1)$ from $(s-1,s)$ (i.e., increases
the first argument by 1). Hence $b_{n+1}=a_n+b_n$.
Thus it is easy to see that $a_n=F_{2n-2}$ and $b_n=F_{2n-3}$ for all $n\ge 1$.
Therefore
\begin{equation*}
  \sharp\sfW^\tI_{\le w}=\sum_{n=0}^{2w-2} F_n=F_{2w}-1
\end{equation*}
which can be proved easily by induction.
\end{proof}

We find up to weight 3 the following identity \eqref{equ:tIwt3missingRel}
cannot be proved by DBSFs and dualities
up to weight 3. Let $1_n$ denote the string where $1$ is repeated $n$ times. Then
\begin{equation}\label{equ:tIwt3missingRel}
\frakz_q^{(1,1)}[2,1]=\frakz_q^{(1,1)}[1,1]-\frakz_q^{(1_3)}[1_3]+\frakz_q^{(1,1,0)}[1_3].
\end{equation}
Interestingly, \eqref{equ:tIwt3missingRel} can be proved using weight 4 DBSFs and dualities.
This is why we put $\bfzero$ as the final deficiency.

\begin{figure}[!h]
{
\begin{center}
\begin{tabular}{|c||c|c|c|c|c|c|c|c|c|}
\hline
{}$w$                                                         & 2 & 3&  4&  5 &  6&   7
\\
\hline
{}$\sharp(\sfW)^\tI_{\le w}$                                  & 4& 12& 33& 88 &232& 609
\\
\hline
lower bound of $\dim{\sfZ^\tI_{\le w}}$                       & 3&  7& 14& 27 & 50&  91
\\
\hline
{}$\dim{\DS^\tI_{\le w}}$                                     & 1&  4& 17& 56 &171& 497
\\
\hline
{}$\dim{\big(\DU^\tI_{\le w}\setminus \DS^\tI_{\le w}\big)}$  & 0&  0& 1 &  2 &  3&   6
\\
\hline
{}deficiency                                                  & 0&1,{\bf 0}&1,{\bf 0}&3&  8 &  15
\\
\hline
\end{tabular}
\end{center}
}
\caption{Dimension of $q$-MZVs of type $\tI$. }
\label{Table:qMZVtI}
\end{figure}
Having proved \eqref{equ:tIwt3missingRel}, we find, up to weight 4, the only one missing relation is
\begin{equation}\label{equ:tIwt4missingRel}
\begin{split}
\frakz_q^{(2,1)}[3,1]= & \frakz_q^{(1,0,1)}[1_3]-2\frakz_q^{(1_3)}[1,2,1]
+\frakz_q^{(1_2,0)}[1,2,1]\\
+& \frakz_q^{(1_3)}[2,1_2]-\frakz_q^{(1,0,1)}[2,1_2]
-\frakz_q^{(1_3,0)}[1_4]+\frakz_q^{(1_2,0_2)}[1_4].
\end{split}
\end{equation}
In weight 5, there are three missing relations:
{\allowdisplaybreaks
\begin{align*}
\frakz_q^{(1_4)}[1_3,2]=&\, \frakz_q^{\bft_2}[1_4] -\frakz_q^{\bft_1}[1_4]-\frakz_q^{(1_4)}[\bfs_1]-\frakz_q^{\bft_2}[1_3,2]-2\frakz_q^{(\bft_1,0)}[1_5]
-2\frakz_q^{(1_3,0)}[\bfs_2]+2\frakz_q^{(\bft_3,1)}[1_5],\\
\frakz_q^{(1_3,0)}[1_4]=&\, \frakz_q^{(1_3)}[2,1_2]-\frakz_q^{\bft_1}[\bfs_1]
   -2\frakz_q^{(1_4)}[\bfs_1]-\frakz_q^{(1,0,1,0)}[\bfs_2] +\frakz_q^{(1_4)}[\bfs_2]+\frakz_q^{(1,0_3,1)}[1_5]\\
-&\frakz_q^{\bft_2}[1_4]-\frakz_q^{(1_3)}[2,1,2]-\frakz_q^{(\bft_3,1)}[1_5]
    +\frakz_q^{\bft_3}[1_4]+2\frakz_q^{\bft_1}[1_4]
-\frakz_q^{\bft_1}[2,1_3]-\frakz_q^{\bft_3}[2,1_3] ,\\
\frakz_q^{(1_4)}[1_3,2]=&\, 3\frakz_q^{(1_3)}[2,1_2]-3\frakz_q^{(1_3)}[2,1,2]
    -3\frakz_q^{(1_3,0)}[1_4]+\frakz_q^{\bft_1}[1_4]-\frakz_q^{(\bft_1,0)}[1_5]\\
-&\frakz_q^{(1_4)}[\bfs_1]-\frakz_q^{(1_3,0)}[\bfs_1]-2\frakz_q^{\bft_1}[\bfs_1]
    -\frakz_q^{(1_3,0_2)}[1_5]-2\frakz_q^{(1_4)}[\bfs_2]+\frakz_q^{(1_3,0)}[\bfs_2]\\
+&\frakz_q^{(1_3,0)}[2,1_3]+2\frakz_q^{\bft_2}[1_4]
   +\frakz_q^{\bft_2}[1_3,2] +2\frakz_q^{(1_2,0_2,1)}[1_5],
\end{align*}}
where $\bfs_1=(1,2,1_2)$, $\bfs_2=(1_2,2,1)$, $\bft_1=(1_2,0,1)$, $\bft_2=(1,0,1_2)$,
and $\bft_3=(1,0_2,1)$.

Equation \eqref{equ:tIwt4missingRel} was initially verified numerically.
Even with all the DBSFs and dualities from weight 5 and 6 this still would not follow.
Fortunately, we will see in a moment that this relation
can be proved using type G $q$-MZVs. However, the three missing
relations in weight 5 are only proved numerically, since, unfortunately,
there are too many type G $q$-MZVs of weight 5 so the
computer computation requires too much memory to provide a solution at the moment.

Using the relations obtained above for type $\tI$ $q$-MZVs
we can compute the following data for type I $q$-MZVs.
\begin{figure}[!h]
{
\begin{center}
\begin{tabular}{|c||c|c|c|c|c|c|c|c|c|c|c|}
\hline
{}$w$                                                      & 2 & 3 & 4 & 5 &  6 & 7 &  8 &  9
\\
\hline
{}$\sharp(\sfW)^\I_{\le w}$                                & 1 & 3 & 7 &15 & 31 & 63&127 &255
\\
\hline
lower bound of $\dim{\sfZ^\I_{\le w}}$                     & 1 & 2 & 4 & 7 & 11 & 18& 27 & 42
\\
\hline
{}$\dim{\DS^\I_{\le w}}$                                   & 0 & 1 & 3 & 8 & 20 & 45 &  &
\\
\hline
{}$\dim{\big(\DU^\I_{\le w}\setminus \DS^\I_{\le w}\big)}$ & 0 & 0 & 0 & 0 &  0 &  0 &  &
\\
\hline
{}deficiency                                               & 0 & 0 & 0 & 0 &  0 &  0 &  &
\\
\hline
\end{tabular}
\end{center}
}
\caption{Dimension of $q$-MZVs of type I.}
\label{Table:qMZVI}
\end{figure}
It is consistent with Takeyama's computation
at the end of \cite{Takeyama2013}. However, our computation shows that the DBSFs from
type $\tI$ $q$-MZVs already imply all the relations among type I $q$-MVZs,
at least when the weight is less than 8. We thus can think these type $\tI$ DBSFs
as ``regularized'' DBSFs for type I $q$-MVZs.
\begin{conj}
All the $\Q$-linear relations of type I $q$-MZVs can be derived by
the regularized DBSFs, i.e., by the DBSFs for type $\tI$ $q$-MZVs.
\end{conj}

\textbf{Type I\!I}. For each fixed weight $w\ge 1$ we collect all the type \II-admissible words of
the following form since we want to use the duality relations to its maximal utility.
Such admissible words must consist of letters $\rho$ and $y$ only, begin with $\rho$, end with $y$,
and the occurrence of $\rho$ and $y$ is at most $w$ each. For example, we have the duality
\begin{equation*}
    \zeta_q^\II(\rho^3 y^2 \rho y^4)=\zeta_q^\II(\rho^4 y \rho^2 y^3)
\Longrightarrow \zeta_q^\II(3,0,1,0_3)=\zeta_q^\II(4,2,0_2)
\end{equation*}
when we consider weight 6.

\begin{prop}\label{prop:numberWII}
For all $w\ge 1$, the number of type \II-admissible words is
\begin{equation*}
\sharp\sfW^\II_{\le w}=\sum_{i=0}^{w-1}\sum_{j=0}^{w-1} \binom{i+j}{j}=\binom{2w}{w}-1.
\end{equation*}
\end{prop}
\begin{rem}
This is the sequence A030662 according to the On-Line Encyclopedia
of Integer Sequences http://oeis.org.
\end{rem}

\begin{proof}
For the first equality, notice that if $i+1$ (resp.\ $j+1$)
is the number of  occurrence of $\rho$  (resp.\ $y$) in an admissible word of $\sfW^\II_{\le w}$
then we can put one $\rho$ at the beginning and one $y$ at the end, then put
$i$ of the other $\rho$'s and $j$ of the other $y$'s in between in arbitrary order. Thus,
by a well-known binomial identity
\begin{multline*}
   1+ \sharp\sfW^\II_{\le w}
=1+\sum_{i=0}^{w-1}  \sum_{j=0}^{w-1} \binom{i+j}{j}
=1+\sum_{i=0}^{w-1} \binom{w+i}{w-1}
=\sum_{i=0}^{w} \binom{w+i-1}{i}
=\binom{2w}{w}.
\end{multline*}
This completes the proof of the proposition.
\end{proof}

\begin{figure}[!h]
{
\begin{center}
\begin{tabular}{|c||c|c|c|c|c|c|c|c|c|}
\hline
{}$w$                                                        & 1 & 2 &  3 &  4 &   5 &   6
\\
\hline
{}$\sharp(\sfW)^\II_{\le w}$                                 & 1 & 5 & 19 & 69 & 251 & 923
\\
\hline
lower bound of $\dim{\sfZ^\II_{\le w}}$                      & 1 & 3 & 12 & 30 &  73 & 173
\\
\hline
{}$\dim{\DS^\II_{\le w}}$                                    & 0 & 1 & 5  & 28 & 124 & 536
\\
\hline
{}$\dim{\big(\DU^\II_{\le w}\setminus \DS^\II_{\le w}\big)}$ & 0 & 1 & 2  &  8 &  35 & 127
\\
\hline
{}deficiency                                                 & 0 & 0 & 0  &3,{\bf 0}&19,{\bf 6}&87
\\
\hline
\end{tabular}
\end{center}
}
\caption{Dimension of $q$-MZVs of type \II.}
\label{Table:qMZVII}
\end{figure}

Up to weight 4, the following three independent relations cannot be proved using DBSFs
and dualities up to weight 4.
\begin{align*}
 \frakz_q^\II[1,0,3]
=&\,\frakz_q^\II[2,2]
+3\frakz_q^\II[1_2,2]
+2\frakz_q^\II[1,0,2,0]
-2\frakz_q^\II[1_2,0,1]\\
+&\,\frakz_q^\II[1_2,0,2]
+\frakz_q^\II[1_2,1,0]
-\frakz_q^\II[1,2,0,1]
+2\frakz_q^\II[2,0,1_2],\\
\frakz_q^\II[3,0]
=&\,\frakz_q^\II[2,2]
-2\frakz_q^\II[3,1]
+\frakz_q^\II[1,0,2,0]
-2\frakz_q^\II[1_2,0,1]
+2\frakz_q^\II[1_2,1,0]\\
-&\,\frakz_q^\II[2,0,2,0]
+\frakz_q^\II[3,0_2,0]
+2\frakz_q^\II[3,0_2,1]
-\frakz_q^\II[3,0,1,0]
+2\frakz_q^\II[3,1,0_2],\\
\frakz_q^\II[1,0,3]
=&\,\frakz_q^\II[2,2]
+2\frakz_q^\II[3,1]
+\frakz_q^\II[1_2,2]
+4\frakz_q^\II[1_2,0,1]
+\frakz_q^\II[1_2,0,2]\\
+&\,\frakz_q^\II[1_2,1,0]
+\frakz_q^\II[1,2,0,1]
+4\frakz_q^\II[2,0,1,0]
+2\frakz_q^\II[2,1,0,1]
+2\frakz_q^\II[2,1_2,0].
\end{align*}
But using DBSFs and dualities in weight 5, these can all be verified.
In weight 5, we have to use the
relations from weight 6 to push the deficiency from 19 down to 6. It is very likely
that relations from weight 7 (or even higher) can reduce this further down to 0. But
our computer runs out of memories so this is not proved.

\textbf{Type I\!I\!I}. The set of type \III-admissible words
$\sfW^\III_{\le w}$ up to weight $w$ consist of those of the form
$\rho^{s_1-1} \pi y \rho^{s_2}y\cdots \rho^{s_d}y$ with $d\le w$,
$|\bfs|\le w$, $s_1\ge 1$ and $s_2,\dots,s_d\ge 0$.
First we have
\begin{prop}
For all $w\ge 1$, we have
\begin{equation*}
\sharp\sfW^\III_{\le w}=\binom{2w}{w}-1.
\end{equation*}
\end{prop}
\begin{proof}
Notice there is an onto map from $\sfW^\III_{\le w}$ to $\sfW^\II_{\le w}$
by changing the all the $\pi$'s to $\rho$. For the inverse map, we can change all the
$\rho$'s to $\pi$ except for the one immediately before the first $y$.
Thus this is a one-to-one correspondence and therefore the proposition follows from
Proposition~\ref{prop:numberWII}.
\end{proof}

We find that the deficiency is not zero when the weight $w=3,4,5,6.$
Moreover, none of these missing $\Q$-linear relations can be
recovered even if we consider all the DBSFs and dualities of weight up to 6.

\begin{figure}[!h]
{
\begin{center}
\begin{tabular}{|c||c|c|c|c|c|c|c|c|c|}
\hline
{}$w$                                                          & 1 & 2 &  3 &  4 &   5 &   6
\\
\hline
{}$\sharp(\sfW)^\III_{\le w}$                                  & 1 & 5 & 19 & 69 & 251 & 923
\\
\hline
lower bound of $\dim{\sfZ^\III_{\le w}}$                       & 1 & 4 & 12 & 30 &  73 & 173
\\
\hline
{}$\dim{\DS^\III_{\le w}}$                                     & 0 & 1 & 5  & 28 & 124 & 536
\\
\hline
{}$\dim{\big(\DU^\III_{\le w}\setminus \DS^\III_{\le w}\big)}$ & 0 & 0 & 1  &  1 &   5 &   4
\\
\hline
{}deficiency                                      & 0 & 0 &1,{\bf 0}&10,{\bf 0}&49,{\bf 6}& 210,{\bf 87}
\\
\hline
\end{tabular}
\end{center}
}
\caption{Dimension of $q$-MZVs of type \III.}
\label{Table:qMZVIII}
\end{figure}
The only missing relation in weight 3 that cannot be proved is
\begin{equation}\label{equ:IIImissingRel}
 \frakz_q^\III[1, 0, 1]=2\frakz_q^\III[1, 1, 0]-\frakz_q^\III[1, 2, 0]
-\frakz_q^\III[2, 0, 0]+\frakz_q^\III[2, 0, 1].
\end{equation}
Up to weight 4 there are 10 missing, up to weight 5, 49, and up to weight 6, 210.
Below, we will see that all of the 10 missing relations up to weight 4
including \eqref{equ:IIImissingRel} can be proved using type G $q$-MZVs.
Similarly, the deficiency up to weight 5 and 6 can be
reduced to 6 and 87, respectively.

\textbf{Type I\!V}. To study type \IV $q$-MZVs  $\frakz_q^{(s_1-1,s_2,\dots, s_d)}[s_1,\dots, s_d]$
we have used the special type \II values $\frakz_q^\II[1,s_2,\dots, s_d]$
to facilitate us (which can be thought as a kind of regularization). Type \IV $q$-MZVs
together with these values have been called type $\tIV$ $q$-MZVs.

\begin{prop}
For all $w\ge 1$, we have
\begin{equation*}
\sharp\sfW^\IV_{\le w}=\binom{2w-1}{w}-1, \qquad
\sharp\sfW^\tIV_{\le w}=\binom{2w}{w}-1.
\end{equation*}
\end{prop}
\begin{rem}
The first number gives the sequence A010763 according to the On-Line Encyclopedia
of Integer Sequences http://oeis.org.
\end{rem}
\begin{proof}
Notice that type \IV-admissible $q$-MZVs are in one-to-one correspondence
to the set $\{(x_1,\dots,x_l)\in(\Z_{\ge 0})^l | x_1+\cdots+x_l=j, 0\le j\le w-2, 1\le l\le w\}$.
For each fixed $j$ we see that the number of nonnegative integer solutions
of $x_1+\cdots+x_l=j$ is given by $\binom{l+j-1}{l-1}$. But
\begin{equation*}
   \sum_{l=1}^{w}  \binom{l+j-1}{l-1}= \binom{w+j}{w-1}
\end{equation*}
by a well-known binomial identity. By the proof similar to that of
Proposition~\ref{prop:numberWII} we see that
\begin{equation*}
\sharp\sfW^\IV_{\le w}=\sum_{j=0}^{w-2} \binom{w+j}{w-1}=\binom{2w-1}{w}-1.
\end{equation*}

For the second equation, we notice that in the word form we have the additional
contribution of the following words: $\rho y$ and
$\rho y \rho^{s_1}y\dots \rho^{s_d}y$, $|\bfs|< w$, $1\le d<w$.
The number of such words is given by ($i$=number of $\rho$'s, $j$=number of $y$'s)
\begin{equation*}
1+\sum_{j=0}^{w-2}\sum_{i=0}^{w-1} \binom{i+j}{i}
=1+\sum_{j=0}^{w-2} \binom{w+j}{w-1}
=1+\sharp\sfW^\IV_{\le w}.
\end{equation*}
Therefore
\begin{equation*}
\sharp\sfW^\IV_{\le w}=1+2\sharp\sfW^\IV_{\le w}=2\binom{2w-1}{w}-1=\binom{2w}{w}-1.
\end{equation*}
The proposition is now proved.
\end{proof}

\begin{figure}[!h]
{
\begin{center}
\begin{tabular}{|c||c|c|c|c|c|c|c|c|c|}
\hline
{}$w$                                                          & 1 & 2 &  3 &  4 &   5 &   6
\\
\hline
{}$\sharp(\sfW)^\tIV_{\le w}$                                  & 1 & 5 & 19 & 69 & 251 & 923
\\
\hline
lower bound of $\dim{\sfZ^\tIV_{\le w}}$                       & 1 & 4 & 12 & 30 &  73 & 173
\\
\hline
{}$\dim{\DS^\tIV_{\le w}}$                                     & 0 & 1 &  5 & 28 & 124 & 536
\\
\hline
{}$\dim{\big(\DU^\tIV_{\le w}\setminus \DS^\tIV_{\le w}\big)}$ & 0 & 0 &  1 &  1 &   4 &   4
\\
\hline
{}deficiency                                       & 0 & 0 &1,{\bf 0}&10,{\bf 0}&50,{\bf 6}& 210,{\bf 87}
\\
\hline
\end{tabular}
\end{center}
}
\caption{Dimension of $q$-MZVs of type $\tIV$.}
\label{Table:qMZVtIV}
\end{figure}

Type $\tIV$ $q$-MZVs are similar to type \II and \III in the sense that
the deficiency is often nonzero, at least when the weight
is less than 6. For example, in weight 3 we have the
following identity which cannot be proved using the
DBSFs and dualities if we only restrict to type $\tIV$
$q$-MZVs of weight and depth no greater than 3.
\begin{equation*}
\frakz_q^\IV[2, 0, 1]
=\frakz_q^\II[1, 0, 1]+\frakz_q^\II[1, 2, 0]
\end{equation*}
However this identity follows from weight 4 DBSFs and dualities.

Comparing Table \ref{Table:qMZVIII} and Table \ref{Table:qMZVtIV}
we observe that there should be some hidden relations between
type \III and $\tIV$ $q$-MZVs. Although the dimensions seem to be the same,
at least for lower weight, the deficiencies are very different.
But using the most general type G values to be considered in a moment, we
can make all the deficiencies smaller.

We can now use all of the relations among type $\tIV$ $q$-MZVs to deduce those for type $\IV$
and collect the data in Table~\ref{Table:qMZVIV}. Furthermore, by converting all
the missing relations using type II values we can reduce all the deficiencies up to weight 5
to 0. For weight 6, using type II values we can only reduce the deficiency from 91 to 56.
It is possible that this can be further reduced to 0 using weight 7 relations of type II values.

\begin{figure}[!h]
{
\begin{center}
\begin{tabular}{|c||c|c|c|c|c|c|c|c|c|}
\hline
{}$w$                                                        & 2 &  3 &  4 & 5  &   6
\\
\hline
{}$\sharp(\sfW)^\IV_{\le w}$                                 & 2 &  9 & 34 &125 & 461
\\
\hline
lower bound of $\dim{\sfZ^\IV_{\le w}}$                      & 2 &  7 & 20 & 55 & 141
\\
\hline
{}$\dim{\DS^\IV_{\le w}}$                                    & 0 &  7 &  9 & 51 &  205
\\
\hline
{}$\dim{\big(\DU^\IV_{\le w}\setminus \DS^\IV_{\le w}\big)}$ & 0 &  0 &  0 &   2 &  24
\\
\hline
{}deficiency                                          & 0 &  0 &5,$\bfzero$ &17,$\bf 0$ & 91,$\bf 56$
\\
\hline
\end{tabular}
\end{center}
}
\caption{Dimension of $q$-MZVs of type \IV.}
\label{Table:qMZVIV}
\end{figure}

\textbf{Type G}. To study the general type \GG $q$-MZVs
$\frakz_q^{(t_1,\dots,t_d)}[s_1,\dots, s_d]$ we need all of the following relations
we have defined so far:
DBSFs, $\bfP$-$\bfR$ and duality relations.

\begin{prop}
For all $w\ge 1$, we have
\begin{equation*}
\sharp\sfW^\G_{\le w}=\sum_{1\le d\le k\le w} \
\sum_{\substack{x_1+\cdots+x_d=d+k-1\\ x_1,\dots,x_d\ge 1}} x_1x_2 \cdots x_d.
\end{equation*}
\end{prop}
\begin{proof}
For each fixed depth $d$ and weight $k\le w$,  let $\frakz_q^\GG(t_1,\dots,t_d)[s_1,\dots,s_d]$
be a type $\GG$-admissible  $q$-MZV
satisfying $s_1+\dots+s_d=k$, $1\le t_1\le s_1, 0\le t_j\le s_j$ for all $j\ge 2$. When
$s_1,\dots,s_d$ are fixed and $t_1,\dots,t_d$ vary, the number of
such values is given by
\begin{equation*}
     s_1(s_2+1)(s_3+1)\cdots (s_d+1).
\end{equation*}
Hence the proposition follows by setting $x_1=s_1,x_2=s_2+1,\dots,x_d=s_d+1.$
\end{proof}

Let $\bfP$-$\bfR^\GG_{\le w}$ be the space generated by all the $\bfP$-$\bfR$ relations
of weight bounded by $w$. Then we see that DBSFs are far from enough and both
$\bfP$-$\bfR$ relations and duality relations contribute non-trivially.
Table~\ref{Table:qMZVG} provides our computational data for the lower weight cases.
\begin{figure}[!h]
{
\begin{center}
\begin{tabular}{|c||c|c|c|c|c|c|c|c|c|}
\hline
{}$w$                                                         & 1 & 2 &  3 &  4 &  5 &   6
\\
\hline
{}$\sharp(\sfW)^\GG_{\le w}$                                  & 1 & 8 & 49 &294 &1791& 11087
\\
\hline
lower bound of $\dim{\sfZ^\GG_{\le w}}$                       & 1 & 4 & 12 & 30 & 73 &173
\\
\hline
{}$\dim{\DS^\GG_{\le w}}$                                     & 0 & 1 &  8 & 76 & $\le 608$    &
\\
\hline
{}$\dim \bfP$-$\bfR^\GG_{\le w}\setminus \big(\DS^\GG_{\le w}
\bigcup \DU^\GG_{\le w}\big)$                                 & 0 & 3 & 27 &177 & $\le 1540$   &
\\
\hline
{}$\dim \DU^\GG_{\le w}\setminus$ $\big(\bfP$-$\bfR^\GG_{\le w}
\bigcup \DS^\GG_{\le w}\big)$                                 & 0 & 0 &  2 &  8 &  $\le 219$  &
\\
\hline
{}deficiency                                                  & 0 & 0 &  0 &  3, $\bfzero$ &     &
\\
\hline
\end{tabular}
\end{center}
}
\caption{Dimension of $q$-MZVs of type $\GG$.}
\label{Table:qMZVG}
\end{figure}
One can see that the number of admissible words increases very fast so that it
is very difficulty to prove relations of other type $q$-MZVs by first finding
all the relations for type G $q$-MZVs.
This is possible theoretically, but not feasible with our current computer powers.

Fortunately, by using $\bfP$-$\bfR$ relations, all the type G $q$-MZVs can be
converted to $\Q$-linear combinations of type \II values. Therefore, the three
missing relations in weight 4 must be provable using weight 5 DBSFs,
$\bfP$-$\bfR$ and duality relations.

Hence, as we expected, the missing relation \eqref{equ:IIImissingRel}
for type \III $q$-MZVs of weight 3 and the 9 missing
relations of weight 4 can now be proved. And furthermore,
the only one missing relation \eqref{equ:tIwt4missingRel}
for type $\tI$ $q$-MZVs of weight 4 can now be proved. We can also
obtain the lower bound of $\dim{\sfZ^\GG_{\le w}}$  from that of type II $q$-MZVs.

\textbf{Type O}. Using Corollary~\ref{cor:qIteratedZetaOk} we may regard Okounkov's
$q$-MZVs as $\Q$-linear combinations of the $q$-MZVs $\frakz_q^\bft[\bfs]$ for
suitable auxiliary variable $\bft$. Further by using the $\bfP$-$\bfR$ relations
we may further reduce this to type \II $q$-MZVs where we don't need the letter $\pi$.
\begin{figure}[!h]
{
\begin{center}
\begin{tabular}{|c||c|c|c|c|c|c|c|c|c|c|c|}
\hline
{}$w$                                                     & 2 & 3 & 4 & 5 &  6&  7&  8&  9& 10& 11& 12
\\
\hline
{}$\sharp(\sfW)^{\rm O}_{\le w}$                          & 1 & 2 & 4 & 7 & 12& 20& 33& 54& 88&143&232
\\
\hline
lower bound of $\dim{\sfZ^{\rm O}_{\le w}}$               & 1 & 2 & 4 & 7 & 11& 18& 27& 42& 63& 95&142
\\
\hline
{}$\dim{\DS^{\rm O}_{\le w}\cup\DU^\GG_{\le w}}$          & 0 & 0 & 0 & 0 &  1&  2&  6& 12&25& 48& 90
\\
\hline
\end{tabular}
\end{center}
}
\caption{Dimension of type O $q$-MZVs, proved rigorously for $w\le 6$ and numerically for $w\le 12$.}
\label{Table:qMZVV}
\end{figure}

Applying the same idea as above it is possible to verify the following Okounkov's dimension
conjecture, at least when the weight is small.

\begin{conj}\label{conj:OkounkovConj}
Let ${\mathbf Z}_w^\O$ be the $\Q$-vector space generated by
$\frakz_q^\O[\bfs]$, $|\bfs|\le w$. Then
\begin{multline*}
 \sum_{w=0}^\infty t^w \dim {\mathbf Z}_{\le w}^\O =\frac{1}{1-t-t^2+t^6+t^8-t^{13}}-\frac{1}{1-t} \\
= t^2+2t^3+4t^4+7t^5+11t^6+18t^7+27t^8+42t^9+63t^{10}+95t^{11}+142t^{12}+O(t^{13}).
\end{multline*}
\end{conj}

For example, we have verified all of the following $\Q$-linearly independent
relations in the lower weight cases up to $q^{100}$, and we can rigorously prove
the first identity \eqref{equ:typeOwt6}
involving only weight 4 and 6 values
by using the relations we have found for type \II $q$-MZVs:
{\allowdisplaybreaks
\begin{align}
4\frakz[6]=&\,\frakz[2,2]+12\frakz[3,3]-6\frakz[4,2], \label{equ:typeOwt6}\\
4\frakz[7]=&\, \frakz[2,3]+\frakz[3,2]+8\frakz[3,4]+6\frakz[4,3]-4\frakz[5,2],  \notag \\
\frakz[8]=&\, \frakz[2,4]-\frakz[6]+2\frakz[3,3]+6\frakz[4,4], \notag \\
9\frakz[8]=&\, \frakz[6]-6\frakz[3,3]+3\frakz[4,2]+20\frakz[3,5]+16\frakz[5,3]-10\frakz[6,2], \notag \\
\frakz[8]=&\, 2\frakz[2,6]-\frakz[6]+2\frakz[3,3]+4\frakz[3,5]-16\frakz[5,3]  \notag \\
-& 6\frakz[2,3,3]+3\frakz[2,4,2]-6\frakz[3,2,3]-3\frakz[4,2,2], \notag \\
4\frakz[3,6]=&\,\frakz[2,5]+4\frakz[5,2]+3\frakz[3,4]+6\frakz[4,5]+8\frakz[5,4]+2\frakz[7,2], \notag \\
8\frakz[9]=&\,\frakz[3,4]-5\frakz[2,5]-8\frakz[5,2]-30\frakz[4,5]-2\frakz[4,3]-36\frakz[5,4]-10\frakz[6,3],\notag \\
6\frakz[4,2]=&\, 10\frakz[6]+42\frakz[8]-60\frakz[2,6]-12\frakz[3,3]-120\frakz[3,5]+312\frakz[5,3]\notag \\
-& 15\frakz[2,2,2]+180\frakz[2,3,3]-90\frakz[2,4,2]+180\frakz[3,2,3]+60\frakz[3,3,2],\notag \\
72\frakz[9]=&\, 62\frakz[5,2]+40\frakz[2,5]-4\frakz[3,4]+40\frakz[3,6]
-2\frakz[4,3]+240\frakz[4,5]+264\frakz[5,4]\notag \\
-&5\frakz[2,2,3]-60\frakz[3,3,3]-30\frakz[4,2,3],\notag \\
16\frakz[9]=&\, 2\frakz[3,4]-10\frakz[2,5]-12\frakz[2,7]-8\frakz[5,2]-60\frakz[4,5]-24\frakz[5,4]\notag \\
+&4\frakz[2,3,2]+4\frakz[3,2,2]+3\frakz[2,2,3]+24\frakz[2,3,4]+18\frakz[2,4,3]\notag \\
+&12\frakz[3,3,3]-12\frakz[2,5,2]+24\frakz[3,2,4]+6\frakz[4,3,2],\notag \\
64\frakz[9]=&\,40\frakz[2,5]+20\frakz[2,7]-8\frakz[3,4]+44\frakz[5,2]+20\frakz[3,6]
-4\frakz[4,3]+240\frakz[4,5]\notag \\
+&168\frakz[5,4]-5\frakz[2,3,2]-5\frakz[2,2,3]-40\frakz[2,3,4]-30\frakz[2,4,3]\notag \\
+&20\frakz[2,5,2]-5\frakz[3,2,2]-40\frakz[3,2,4]-100\frakz[3,3,3]+10\frakz[3,4,2], \notag \\
56\frakz[9]=&\,30\frakz[2,5]+20\frakz[2,7]+26\frakz[5,2]-\frakz[3,4]
+40\frakz[3,6]-6\frakz[4,3]\notag \\
+&180\frakz[4,5]+112\frakz[5,4]
-5\frakz[2,2,3]-5\frakz[2,3,2]-5\frakz[3,2,2]-40\frakz[2,3,4]\notag \\
+&20\frakz[5,2,2]-40\frakz[3,2,4]-30\frakz[2,4,3]+20\frakz[2,5,2]-140\frakz[3,3,3].\notag
\end{align}}
Therefore, Conjecture~\ref{conj:OkounkovConj} is proved rigorously up to weight 6 (inclusive),
and verified numerically up to weight 12 (inclusive). The list of relations for weight 10 to 12
is too long to be presented here.

\section{Conclusions}
We have studied various $q$-analogs of MZVs in this paper using the
uniform method of Rota-Baxter algebras. Among these $q$-MZVs, there
are many $\Q$-linear relations, most of which can be proved using
DBSFs, $\bfP$-$\bfR$ and duality relations.

{}From the data collected in section \ref{sec:data}, we have seen that for all
of the type $\tI$, \II, \III and $\tIV$ $q$-MZVs
duality relations are necessary to generate some $\Q$-linear
relations among $q$-MZVs that are missed by the DBSFs,
at least when the weight is large enough. However, the
combination of all the DBSFs and dualities are often not exhaustive yet.
Sometimes, this difficulty can be overcome by increasing the weight and depth.
But this seems to fail in some other cases, for example,
for type $\tI$ $q$-MZVs of weight 4.

We can improve the above situation by considering the more general type G values.
The advantage is that we have the new $\bfP$-$\bfR$ relations which
provide a lot of new relations between type G $q$-MZVs, much more than
the DBSFs and duality combined. The disadvantage is that there are
too many type G values so that even when the weight is 5 our computer
power is too week to produce all the necessary relations. However,
by using $\bfP$-$\bfR$ relations all type G values can be converted to
$\Q$-linear combinations of type \II values which can be handled by
computer a lot easier.

As we mentioned in the introduction our method can be easily adapted to study
the $q$-MZVs of the following forms:
\begin{equation*}
\frakz_q^{(a_1,\dots, a_d)}[s_1,\dots, s_d], \qquad
\frakz_q^{(s_1-a_1,\dots, s_d-a_d)}[s_1,\dots, s_d],
\end{equation*}
where $a_1\ge a_2 \ge \cdots\ge a_d\ge 0$ are all integers.
The monotonicity guarantees that
a good stuffle structure can be defined.
For $\frakz_q^{(a_1,\dots, a_d)}[s_1,\dots, s_d]$ we need to use embedding (A)
together with shifting operator $\sif_{-}$ in defining the stuffle and,
for $\frakz_q^{(s_1-a_1,\dots, s_d-a_d)}[s_1,\dots, s_d]$, use (B) together with $\sif_{+}$.

As an application, we are able to prove Okounkov's Conjecture~\ref{conj:OkounkovConj} rigorously
up to weight 6 (inclusive), and verify it numerically up to weight 12 (inclusive).
It would be more effective if one can define a shuffle structure for type O values
themselves and find a relation to the differential operator $q\frac{d}{dq}$ which should play
an important role in the study of these vales.

\medskip
\noindent
\textbf{Acknowledgement.} This work,
supported by NSF grant DMS-1162116, was done while the author was visiting
Max Planck Institute for Mathematics and ICMAT at Madrid, Spain. He is
very grateful to both institutions for their hospitality and support. He
also would like to thank Kurusch Ebrahimi-Fard for the enlightening conversations
and his detailed explanation of their joint paper \cite{CEM2013}.


\begin{thebibliography}{99}
\bibitem{BachmannKu2013}
H.\ Bachmann and U.\ K\"uhn, The algebra of multiple divisor functions and
applications to multiple zeta values, arXiv:1309.3920.

\bibitem{BachmannKu2014}
H.\ Bachmann and U.\ K\"uhn, A short note on a conjecture of Okounkov about a
$q$-analogue of multiple zeta values, arXiv:1407.6796.

\bibitem{Bradley2005}
D.\ M.\ Bradley,  Multiple $q$-zeta values,
\emph{J.\ of Algebra} \textbf{283} (2005), pp.\ 752--798, arXiv:math/0402093

\bibitem{Broadhurst1996}
D.J.\ Broadhurst, Conjectured enumeration of
irreducible multiple zeta values, from knots and Feynman diagrams,
arXiv:hep-th/9612012.

\bibitem{Brown2012}
F.\ Brown, Mixed Tate motives over Spec($\Z$),
\emph{Ann.\ Math.} \textbf{175}(2) (2012), pp.\ 949--976.

\bibitem{CEM2013a}
J.\ Castillo Medina, K.\ Ebrahimi-Fard, D.\ Manchon,
On Euler's decomposition formula for $q$MZVs, arXiv:1309.2759.

\bibitem{Chen1971}
K.-T.\ Chen, Algebras of iterated path integrals and fundamental groups,
\emph{Trans.\ Amer.\ Math.\ Soc.} \textbf{156} (1971), pp.\ 359--379.

\bibitem{CEM2013}
J.\ Castillo Medina, K.\ Ebrahimi-Fard, D.\ Manchon,
Unfolding the double shuffle structure of $q$-multiple zeta values,
arXiv:1310.1330

\bibitem{Euler1775}
L.\ Euler, Meditationes circa singulare serierum genus,
\emph{Novi Comm.\ Acad.\ Sci.\ Petropol.}\ \textbf{20} (1776), pp.\ 140--186;
reprinted in Opera Omnia, Ser.\ I, Vol.\ \textbf{15}, B.\ Teubner, Berlin, 1927, pp.\ 217--267.

\bibitem{GoncharovMa2004}
A.B.\ Goncharov and Y.I.\ Manin, Multiple $\zeta$-motives and moduli spaces $M_{0,n}$,
\emph{Compositio Math.} \textbf{140} (2004), pp.\ 1--14.

\bibitem{GuoKe2008}
	L.~Guo and W.~Keigher,
	{On differential Rota--Baxter algebras},
	\emph{J.\ Pure and Appl.\ Alg.}, \textbf{212} (2008), pp.\ 522--540.

\bibitem{Hoffman1992}
M.E.\ Hoffman, Multiple harmonic series,
\emph{Pacific J.\ Math.} \textbf{152} (1992), pp.\ 275--290.

\bibitem{Hoffman1997}
M.E.\ Hoffman, The algebra of multiple harmonic series,
\emph{J.\ Alg.} \textbf{194} (1997), pp.\ 477-–495.

\bibitem{Hoffman2000}
M.E.\ Hoffman, Quasi-shuffle products,
\emph{J.\ Alg.\ Combin.} \textbf{11} (2000), pp.\ 49--68.

\bibitem{IKZ2006}
K.\ Ihara, M.\ Kaneko and D.\ Zagier, Derivation and double shuffle
relations for multiple zeta values,
\emph{Compos.\ Math.} \textbf{142} (2006), pp.\ 307--338.

\bibitem{KanekoKuWa2003}
M.\ Kaneko, N.\ Kurokawa and M.\ Wakayama, A variation of Euler's approach to values
of the Riemann zeta function, \emph{Kyushu J.\ Math.} \textbf{57} (2003), pp.\ 175--192.
arXiv:math/0206024.

\bibitem{KurokawaLaOc2009}
N.\ Kurokawa, M.\ Lalin, and H.\ Ochiai,
Higher Mahler measures and zeta functions, \emph{Acta Arith.} \textbf{135}(3) (2008), pp.\ 269--297.

\bibitem{LeMu1995}
T.Q.T.\ Le and J.\ Murakami, Kontsevich's integral for the
Homfly polynomial and relations between values of multiple zeta functions,
\emph{Topology Appl.} \textbf{62} (1995), pp.\ 193--206.

\bibitem{OOZ2012}
Y.\ Ohno, J.\ Okuda and W.\ Zudilin, Cyclic $q$-MZSV sum,
\emph{J.\ Number Theory} \textbf{132} (2012), pp.\ 144--155.

\bibitem{Okounkov2014}
A.\ Okounkov, Hilbert schemes and multiple $q$-zeta values,
arXiv:1404.3873.

\bibitem{OkudaTa2007}
J.\ Okuda and Y.\ Takeyama, On relations for the $q$-multiple zeta values,
\emph{Ramanujan J.} \textbf{14} (2007), pp.\ 379--387

\bibitem{Schlesinger2001}
K.-G.\ Schlesinger, Some remarks on $q$-deformed multiple polylogarithms,
arXiv:math/0111022.

\bibitem{Takeyama2013}
Y.\ Takeyama, The algebra of a $q$-analogue of multiple harmonic series,
\emph{SIGMA} \textbf{9} (2013), Paper 0601, 15 pp.

\bibitem{Terasoma2006}
T.\ Terasoma, Geometry of multiple zeta values,  in: \emph{Proc.\ Intl.\ Congress of Mathematicians}
(Madrid, 2006), Vol.\ \textbf{II}, European Mathe.\ Soc., Z\"urich, 2006, pp.\ 627--635.

\bibitem{Zagier1994}
D.\ Zagier, Values of zeta functions and their applications,
in: First European Congress of Mathematics (Paris, 1992), Vol.\ \textbf{II},
A.\ Joseph et al.\ (eds.), Birkh\"auser, Basel, 1994, pp.\ 497--512.

\bibitem{Zhao2007c}
J.\ Zhao, Multiple $q$-zeta functions and multiple $q$-polylogarithms,
\emph{Ramanujan J.},  \textbf{14}(2) (2007), pp.\ 189--221, arXiv:math/0304448.

\bibitem{Zudilin2003}
W.\ Zudilin, Algebraic relations for multiple zeta values,
\emph{Russian Math.\ Surveys} \textbf{58}(1) (2003), 1--29.

\bibitem{Zudilin2014}
W.\ Zudilin,
Multiple $q$-zeta brackets, arXiv:1412.0163.

\end{thebibliography}
\end{document}